\documentclass{amsart}

\usepackage{amsfonts}
\usepackage{amssymb}
\usepackage{hyperref}
\usepackage{bm}
\usepackage{color}

	\usepackage{tikz} 
\usetikzlibrary{trees,arrows}
\usetikzlibrary{patterns}
\usetikzlibrary{positioning}
\tikzset{mynode/.style={draw,text width=4cm,align=center}
}

\usepackage{tkz-fct,tkz-euclide,tikz-layers}
\usepackage{animate}

\tikzset{arrow coord style/.style={dotted, opacity=.8, thin}}
\tikzset{xcoord style/.style={
		font=\footnotesize,text height=1ex,
		inner sep = 0pt,
		outer sep = 0pt,
		text=black}}
\tikzset{ycoord style/.style={
		font=\footnotesize,text height=1ex,
		inner sep = 0pt,
		outer sep = 0pt,
		text=black}}

\newtheorem{theorem}{Theorem}[section]
\newtheorem{lemma}[theorem]{Lemma}
\theoremstyle{definition}
\newtheorem{definition}[theorem]{Definition}
\newtheorem{proposition}[theorem]{Proposition}
\newtheorem{corollary}[theorem]{Corollary}

\newtheorem{problem}[theorem]{Cauchy Problem}
\newtheorem{notation}[theorem]{Notation}
\newtheorem{assumption}[theorem]{Assumption}
\theoremstyle{remark}
\newtheorem{remark}[theorem]{Remark}
\numberwithin{equation}{section}

\newcommand{\BR}{{\mathbb R}}
\newcommand{\BC}{{\mathbb C}}
\newcommand{\BN}{{\mathbb N}}

\begin{document}

\title[Structurally damped $\sigma-$evolution equations with power-law memory]
{Structurally damped $\sigma-$evolution equations with power-law memory}

\author{{N.~Faustino}}
\address[\textsc{N.~Faustino}]{Department of Mathematics, University of Aveiro, Campus Universit\'ario de Santiago, 3810-193 Aveiro, Portugal}
\email{\href{mailto:nfaust@ua.pt}{nfaust@ua.pt}}
\address[\textsc{N.~Faustino}]{Center for R\&D in Mathematics and Applications (CIDMA), Department of Mathematics, University of Aveiro, Campus Universit\'ario de Santiago, 3810-193 Aveiro, Portugal}
\email{\href{mailto:nelson.faustino@ymail.com}{nelson.faustino@ymail.com}}
\thanks{\href{http://orcid.org/0000-0002-9117-2021}{N.~Faustino} was supported by The Center for Research and Development in Mathematics and Applications (CIDMA) through the Portuguese Foundation for Science and Technology (FCT), references UIDB/04106/2020 and UIDP/04106/2020.}

\author{{J.~Marques}}
\address[\textsc{J. Marques}]{Faculty of Economics and CeBER, University of Coimbra,
	Av. Dias da Silva, 165,
	3004-512 Coimbra,
	Portugal}
\email{\href{mailto:jmarques@fe.uc.pt}{jmarques@fe.uc.pt}}
\thanks{\href{https://orcid.org/0000-0003-3825-6389}{J.~Marques} was supported by Centre for Business and Economics Research (CeBER) through the Portuguese Foundation for Science and Technology (FCT), reference UIDB/05037/2020}

\date{\today}

\begin{abstract}
We
consider an integro-differential counterpart of the $\sigma-$evolution equation of the type
\[
\partial_t^2 u(t,x)+\mu (-\Delta)^{\frac{\sigma}{2}} \partial_t u(t,x)+(-\Delta)^\sigma u(t,x)=f(t,x),
\]
with $\sigma>0$ and $\mu>0$, that encodes memory of \textit{power-law} type. To do so, we replace the time derivatives $\partial_t$ and $\partial_t^2$ by the so-called Caputo-Djrbashian derivatives $\partial_t^\gamma$ of order $\gamma=\alpha$ and $\gamma=2\alpha$, respectively, and the inhomogeneous term $f(t,x)$ by the Riemann-Liouville integral $I^{\beta-2\alpha}_{0^+}f(t,x)$, whereby $0<\alpha\leq 1$ and $2\alpha\leq \beta<2\alpha+1$.
For the solution representation of the underlying Cauchy problems on the space-time $[0,T]\times \mathbb{R}^n$ we then consider a wide class of pseudo-differential operators $\displaystyle (-\Delta)^{\frac{\eta}{2}}E_{\alpha,\beta}\left(~-\lambda(-\Delta)^{\frac{\sigma}{2}} t^\alpha~\right)$, endowed by the  fractional Laplacian $-(-\Delta)^{\frac{\sigma}{2}}$ and the two-parameter Mittag-Leffler functions $E_{\alpha,\beta}$. 
On our approach we are also able to provide dispersive and Strichartz estimates for the solutions with the aid of decay properties of $E_{\alpha,\beta}(-z)$ ($z\in \mathbb{C}$) and the boundedness properties of the Hankel transform.

\end{abstract}

\subjclass[2020]{26A33, 33C10, 33E12, 35R11, 42B37, 44A20}

\keywords{Fractional differential equations, Mittag-Leffler functions, structural damping}

\maketitle

\tableofcontents

\section{Introduction}

Let us begin with a few facts concerning the theory of {\it structurally damped equations}.
After Russell's fundamental paper \cite{Russell86} on mathematical models of structural mechanics describing elastic beams much more attention has been paid rightly to the strongly damped plate/wave problems and to its further applications on the crossroads of control theory (cf.~\cite{AN10}), dynamical systems (cf.~\cite{Wang14}) and transmition problems (cf.
\cite{MDenkMK19}).
The current trends on {\it struturally damped equations}, exhibited on the contribution of Pham et al \cite{PMR15}, and more recently on D'Abbicco-Ebert's series of papers \cite{AE17,AE_21,AE22}, paved the way for considering space-fractional model problems such as the {inhomogeneous} $\sigma-$evolution equation
\begin{equation} 
	\label{StructuralDamping_integer}
	\partial_t^2 (t,x)+\mu (-\Delta)^{\frac{\sigma}{2}} \partial_t u(t,x)+(-\Delta)^\sigma u(t,x)=f(t,x).
\end{equation}

We refer the reader to the works of Carvalho et al (cf.~\cite{CCD08}) and Denk-Schnaubelt (cf.~\cite{DS15}) just to mention a few papers on analysis of PDEs towards struturally damped plate equations (case of $\sigma=2$) and damped wave equations (case of $\sigma=1$). Also, we refer the seminal work of Karch (cf.~\cite{karch2000selfsimilar}) for a general account on the asymptotic behaviour of structurally damped models.

We notice that the $\sigma-$evolution equation (\ref{StructuralDamping_integer}) describes the scenario where the higher frequencies exhibit a strongly damped behavior in comparison with the lower frequencies. Here, the fractional differential operator $(-\Delta)^{\frac{\gamma}{2}}$ (that will be defined later on Subsection \ref{FourierAnalysisSetup}) is used to model the damping term $(\gamma=\sigma)$ and the elastic term as well $(\gamma=2\sigma)$. Noteworthy, in the view of Carvalho et al (cf.~\cite{CCD08}) and Chill-Srivastava (cf.~\cite{CS05}), the model problem (\ref{StructuralDamping_integer}) brings also the possibility to study the well-posedness of problems carrying $L^p$ data ($1<p<\infty$), in stark constrast with the {\it weaker damping case} so that $\mu (-\Delta)^{\frac{\sigma}{2}} \partial_t u(t,x)$ $(\mu>0)$ resembles to a regularization of the weak damping term $\mu \partial_t u(t,x)$, in the limit $\sigma \rightarrow 0^+$.

Our aim here is to investigate a space-time-fractional counterpart of (\ref{StructuralDamping_integer}) that encodes memory effects, ubiquitous in natural phenomena (cf.~\cite{DWH13,Stinga22}). In our setting it will be replaced the time derivatives $\partial_t$ resp. $\partial_t^2$ by the so-called Caputo-Djrbashian derivatives $ \partial^{\gamma}_t$ of order $\gamma$, defined as
\begin{eqnarray}
	\label{CaputoDerivative} \partial^\gamma_t u(t,x):=\left\{\begin{array}{lll} \displaystyle
		\int_0^t\dfrac{(t-\tau)^{m-\gamma-1}}{\Gamma(m-\gamma)}\frac{\partial^m u(\tau,x)}{\partial \tau^m }d\tau &, m-1<\gamma <m
		\\ \ \\
		\displaystyle \frac{\partial^mu(t,x)}{\partial t^m} &, \gamma=m,
	\end{array}\right.
\end{eqnarray}
where $\Gamma(\cdot)$ denotes the Euler's Gamma function (see eq.~(\ref{GammaFunction}) of Subsection \ref{MittagLefflerFRAC}) and $m=\lfloor \gamma \rfloor+1$ ($\lfloor \gamma \rfloor$ stands for the integer part of $\gamma$). 

For the time-fractional counterpart of the inhomogeneous term $f(t,x)$, it will be adopted the Riemann-Liouville integral
\begin{eqnarray}
	\label{RiemannLiouville}  I^{\gamma}_{0^+} f(t,x):=\left\{\begin{array}{lll} \displaystyle
		\int_0^t\dfrac{(t-\tau)^{\gamma-1}}{\Gamma(\gamma)}f(\tau,x)d\tau &, 0<\gamma <1
		\\ \ \\
		\displaystyle f(t,x) &, \gamma=0.
	\end{array}\right.
\end{eqnarray}

We observe that for values of $\gamma\not \in \BN_0$ resp. $\gamma\neq 0$, the eqs. (\ref{CaputoDerivative}) and (\ref{RiemannLiouville}) always define operators with power-law memory. The underlying memory function, given by 
\begin{eqnarray}
	\label{MemoryFunction}
	\displaystyle g_{\nu}(t):=\dfrac{t^{\nu-1}}{\Gamma(\nu)}~~(0<\nu<1),
\end{eqnarray} is a probability density function over the interval $(0,\infty)$, converging to the delta function $\delta(t)$ in the limit $\nu\rightarrow 0^+$.

Some of the physical reasons for the choice of the Caputo-Djrbashian derivative (\ref{CaputoDerivative}) instead of the Riemann-Liouville derivative ${~}^{RL}\partial_t^\gamma:=\partial_t^mI^{m-\gamma}_{0^+}$ stems on its ubiquity in the generalized Langevin equations (cf.~\cite{LLL17}) and on certain limiting processes (cf. \cite{HairerEtAl18}). Despite their usefulness on real-world applications (cf.~\cite{ABM16,Almeida17}), the Caputo-Djrbashian derivative is suitable for modelling Cauchy problems for the two amongst many reasons:  
\begin{enumerate}
	\item they have similar properties to the time-derivatives;
	\item they remove the singularities at $t=0$.
\end{enumerate}

We refer to the book of Samko et al. (cf.~\cite{samko1993fractional}) for further details on the theory of Caputo-Djrbashian derivatives and Riemann-Liouville operators and Stinga's preprint \cite{Stinga22} for an abridged overview of it. For its applications on Cauchy problems, we refer to the books of Podlubny (cf.~\cite{Podlubny99}) and Kilbas et al. (cf.~\cite{KST06}). We also refer to \cite{KSVZ16,KSZ17,AEP19,LVaz20,DAbbiccoGirardi22} for recent applications on fractional diffusion equations, closely related to our approach.

\section{Discussion}\label{Discussion}

\subsection{Model Problems}

The main focus in this paper are the following two Cauchy problems, carrying the initial condition(s) $u_0(x)$ [and $u_1(x)$] and the inhomogeneous term $I^{\beta-2\alpha}_{0^+}f(t,x)$, fulfilling certain regularity conditions to be presented a posteriori. Here and elsewhere $I^{\beta-2\alpha}_{0^+}$ stands for the Riemann-Liouville integral, defined through eq. (\ref{RiemannLiouville}) for $\gamma=\beta-2\alpha$.

\begin{problem}\label{CP_parabolic}
	Given $0<\alpha\leq \frac{1}{2}$, $2\alpha\leq \beta< 2\alpha+1$, $\sigma>0$ and $\mu>0$, the function $u:[0,T]\times \BR^n\rightarrow \BR$ is a (weak) solution of
	\begin{equation*}
		\begin{cases}
			\displaystyle \partial^{2\alpha}_t u(t,x)+\mu (-\Delta)^{\frac{\sigma}{2}}\partial^\alpha_t u(t,x)+(-\Delta)^\sigma u(t,x)=I^{\beta-2\alpha}_{0^+}f(t,x)&,\mbox{in}~(0,T]\times \mathbb{R}^n\\ \ \\
			u(0,x)=u_0(x)~&,\mbox{in}~\mathbb{R}^n.
		\end{cases}
	\end{equation*}
\end{problem}

\begin{problem}\label{CP_hyperbolic}
	Given $\frac{1}{2}<\alpha\leq 1$, $2\alpha\leq\beta< 2\alpha+1$, $\sigma>0$ and $\mu>0$, the function $u:[0,T]\times \BR^n\rightarrow \BR$ is a (weak) solution of
	\begin{equation*}
		\begin{cases}
			\displaystyle 	\partial^{2\alpha}_t u(t,x)+\mu (-\Delta)^{\frac{\sigma}{2}}\partial^\alpha_t u(t,x)+(-\Delta)^\sigma u(t,x)=I^{\beta-2\alpha}_{0^+}f(t,x) &,\mbox{in}~(0,T] \times \mathbb{R}^n\\ \ \\
			u(0,x)=u_0(x)&,\mbox{in}~\mathbb{R}^n\\ \ \\
			\partial_t u(0,x)=u_1(x)&,\mbox{in}~\mathbb{R}^n.
		\end{cases}
	\end{equation*}
\end{problem}

The programme of studying Cauchy problems similar to {\bf Cauchy Problem \ref{CP_parabolic}} \& {\bf Cauchy Problem \ref{CP_hyperbolic}} has been started by Fujita in the former papers
\cite{Fujita90,FujitaII90}, where it has been shown that fractional diffusion and wave equations can be reformulated as integro-differential equations of D'Alembert type. In our case, the proposed model problems amalgamate the following properties, tactically investigated on several works (see e.g.  \cite{CCD08,DS15,PMR15,KSVZ16,KSZ17}):
\begin{enumerate}
	\item The damping part has ‘half of the order’ of the leading elastic term, which is in accordance with Russell's formulation \cite{Russell86};
	\item For values of $0<\alpha<1$, the time-fractional derivatives $\partial^{2\alpha}_t$ and $\partial^{\alpha}_t$ endow a fractional diffusion process;
	\item In the limit $\mu\rightarrow 0^+$, the {\bf Cauchy Problem \ref{CP_parabolic}} \& {\bf Cauchy Problem \ref{CP_hyperbolic}} approaches the {\it undamped case};
	\item For $\alpha=1$ and $1<\sigma<2$, our model problem possess mixed properties of
	the plate equation (case of $\sigma=2$) resp. the wave equation (case of $\sigma=1$) with structural damping;
	\item For $\alpha=1$, the asymptotic profile of the homogeneous part of the {\bf Cauchy Problem \ref{CP_hyperbolic}} changes accordingly to the size of the damping parameter $\mu>0$.
\end{enumerate}

From the point of view of potential applications, our model problem brings also the possibility to model anomalous diffusions and memory effects in a twofold way. Namely, the power law proportional to $t^{2\alpha}$ measures if our diffusive process is subdiffusive $(0<2\alpha<1)$ or superdiffusive ($1<2\alpha<2$). On both cases, the fractional derivatives $\partial_t^{\alpha}$ and $\partial_t^{2\alpha}$ are described in terms of the memory functions $g_{1-\alpha}(t)$ and $g_{2-2\alpha}(t)$, respectively (see eq. (\ref{MemoryFunction})).  

On the other hand, the choice of the parameter $\beta$, appearing on the Riemann-Liouville integral $I^{\beta-2\alpha}_{0^+}f(t,x)$, reflects if our {\it inhomogeneous term} exhibits a {\it memoryless pattern} or not: The limit case $\beta\rightarrow (2\alpha)^+$ guarantees that the shape of the memory function $g_{\beta-2\alpha}(t)$ is close to the shape of the delta function $\delta(t)$ so that the {\it inhomogeneous term} approaches 'lack of memory' for values of $\beta$ close to $2\alpha$. For values of $\beta$ a bit far from $2\alpha$, the ubiquity of memory effects is manifested by the polynomial shape of the graph of $g_{\beta-2\alpha}(t)$. Interesting to see, the particular choice $\beta=1+\alpha$ ($\beta-2\alpha=1-\alpha$), with $0<\alpha<1$, lead us always to diffusive processes with memory. 

Such subtle details encoded on our model may be seen as surprising and substantial from the vast majority of papers available on the literature. Up to our knowledge, the overlap between the memory effects of the fractional derivatives and the {\it inhomogeneous term} are not so often considered. 

\subsection{Solution Representation}

From an operational calculus perpective, the {\bf Cauchy Problem \ref{CP_parabolic}} \& {\bf Cauchy Problem \ref{CP_hyperbolic}} share the same features of Cauchy problems underlying to the time-fractional telegraph equations so that one can exploit, as in \cite{LVaz20}, Orsingher-Beghin's approach \cite{OrsingherBeghin04} to obtain the underlying solution representations (see also ref.~\cite{LLPMeersch19}). Namely, the 
closed-form representation for the solutions of both Cauchy problems -- to be obtained in Subsection \ref{SolutionCP_parabolic} and Subsection \ref{SolutionCP_hyperbolic}, respectively -- follows from the fact that Fourier-Laplace multiplier of $\displaystyle \partial^{2\alpha}_t +\mu (-\Delta)^{\frac{\sigma}{2}}\partial^\alpha_t+(-\Delta)^\sigma$, $$(s^\alpha,|\xi|^\sigma)\mapsto s^{2\alpha}+\mu |\xi|^{\sigma}s^\alpha+|\xi|^{2\sigma},$$
is a quadratic form on the variables $u=s^\alpha$ and $v=|\xi|^\sigma$.

Then, with the aid operational properties of the generalized Mittag-Leffler functions $E_{\alpha,\beta}^\gamma$ and $E_{\alpha,\beta}=E_{\alpha,\beta}^1$, highlighted on Section \ref{DispersiveSection} -- see, for instance, eq.~(\ref{LaplaceIdentityMittagLeffler}) and Theorem \ref{Properties_ML} -- we are able to recast closed representation formulae in terms of the wide class of pseudo-differential operators $\displaystyle (-\Delta)^{\frac{\eta}{2}}E_{\alpha,\beta}\left(~-\lambda (-\Delta)^{\frac{\sigma}{2}}t^\alpha~\right)$. From a semigroup perspective,
such class provides us fractional analogues of the resolvent operator considered e.g. on \cite{CCD08,CS05,DS15} for the case where the damping operator is a multiple square root of the elastic operator (cf.~\cite[Chapter VI]{engel2000one}).  

By the properties of the generalized Mittag-Leffler functions, it will be also shown in Subsection \ref{SolutionCP_parabolic} and Subsection \ref{SolutionCP_hyperbolic} that the choice of $I^{\beta-2\alpha}_{0^+}f(t,x)$ as {\it inhomogeneous term} turns out to be linked with the Fourier multipliers $\displaystyle t^{\beta-1}E_{\alpha,\beta}^{\gamma}(-\lambda|\xi|^\sigma t^{\alpha})$ ($\gamma=1,2$). From feasible physical descriptions of relaxation processes, such class of functions is somewhat expected, as highlighted on the seminal paper \cite{CapelasMainardiVaz11}; for an overview, we refer to \cite[Subsection 8.1.]{GKMaiR14}.

As will be noticed by the reader, the results enclosed on this paper can certainly be rewritten in terms of the language of fractional resolvent families (see e.g. \cite{LLPMeersch19} and the references therein) in a way that the Banach space framework, considered e.g. in the papers \cite{CS05,CCD08,DS15} and on the book \cite{engel2000one} can be faithfully generalized. But we have decided to use mostly pseudo-differential calculus, rather than operator theory, to keep the approach as self-contained as possible.

\subsection{Dispersive and Strichartz estimates}

Apart from the solution representation of these Cauchy problems being interesting from the point of view of spectral analysis and stochastic processes, we want to push forward the investigation of dispersive and Strichartz estimates, with the aim of measuring the size and the decay of the solutions of {\bf Cauchy Problem \ref{CP_parabolic}} \& {\bf Cauchy Problem \ref{CP_hyperbolic}}. The study of such type of estimates have become a cornerstone tool for many years now, mainly due to the Keel-Tao's breakthrough contribution to the topic (cf.~\cite{KTao98}); see also~\cite[Chapter 2.]{Tao06} for an overview.

The starting point will be the decay properties of the convolution kernel of $\displaystyle (-\Delta)^{\frac{\eta}{2}}E_{\alpha,\beta}\left(~-\lambda (-\Delta)^{\frac{\sigma}{2}}t^\alpha~\right)$, represented in terms of the Hankel transform of the Fourier multiplier $\displaystyle |\xi|^{\eta}E_{\alpha,\beta}\left(-\lambda |\xi|^{\sigma}t^\alpha\right)$. To do so, we will exploit the [optimal] decay estimates for the Mittag-Leffler functions $E_{\alpha,\beta}(-z)$, available on the book \cite{Podlubny99} (case of $z\in \mathbb{C}$) and on the papers \cite{Simon14,BS21} (case of $z\geq 0$) to our setting.

Concerning the dispersive estimates, the framework enclosed in Section \ref{DispersiveSection} is comparable to Pham et al and D'Abbicco-Ebert approaches (cf.~\cite{PMR15,AE_21}), who have investigated $L^p-L^q$ estimates for Cauchy problems similar to {\bf Cauchy Problem \ref{CP_hyperbolic}}.
Regarding the Strichartz estimates enclosed in the end of Section \ref{MainSection}, we have followed up some of the ideas addressed by Tao on the monograph \cite{Tao06}, properly adapted to our model problems. Mainly, a version of the so-called {\it Christ-Kiselev lemma} (cf.~\cite[Lemma 2.4]{Tao06}) is always required to ensure the boundedness of the inhomogeneous part of the solution representation of both Cauchy problems (see also \cite[THEOREM 1.2.]{KTao98}).

\section{Organization of the Paper}

We organize the remaining parts of this paper as follows:
\begin{itemize}
	\item In Section \ref{Preliminaries} we collect some results to be employed in the subsequent sections. Namely, some of the properties involving the Laplace transform and the generalized Mittag-Leffler functions $E_{\alpha,\beta}$ and $E_{\alpha,\beta}^\gamma$ are recalled in Subsection \ref{MittagLefflerFRAC}. And in Subsection \ref{FourierAnalysisSetup} we introduce some relevant definitions and properties to define properly the space-fractional differential operator $(-\Delta)^{\frac{\gamma}{2}}$ and the associated function spaces. Special emphasize will be given to the {\it homogeneous Sobolev spaces} $\dot{W}^{\gamma,p}$.
	\item In Section {\ref{DispersiveSection}} we deduce dispersive estimates for the wide class of pseudo-differential operators $(-\Delta)^{\frac{\eta}{2}}E_{\alpha,\beta}\left(~-\lambda(-\Delta)^{\frac{\sigma}{2}} t^\alpha~\right)$. The results proved in Subsection \ref{LpLqEstimates}, Subsection \ref{SobolevEstimates} and Subsection \ref{Overview} describe the decay properties of it. Here, we make use of the reformulation of the Fourier transform for radially symmetric functions in terms of the Hankel transform to obtain, in case of $\lambda\geq 0$, a sharp control of decay of the underlying convolution kernel, after application of Young's inequality. 
	\item
	In Section \ref{MainSection} we provide the main results of this paper. More precisely, we obtain in Subsection \ref{SolutionCP_parabolic} and Subsection \ref{SolutionCP_hyperbolic} closed-form representations for the {\bf Cauchy Problem \ref{CP_parabolic}} \& {\bf Cauchy Problem \ref{CP_hyperbolic}}, respectively. And in Subsection \ref{StrichartzSection} we turn to the study of Strichartz estimates on the {\it mixed-normed Lebesgue spaces} $L_t^sL_x^q([0,T]\times \BR^n)$ for both Cauchy problems, from the dispersive estimates obtained in Section \ref{DispersiveSection}.
	\item Finally, in Section \ref{Conclusions} we address the main conclusions of our approach. Some further remarks and open problems will also be highlighted.
\end{itemize}

\section{Preliminaries}\label{Preliminaries}

\subsection{Laplace Transform and Mittag-Leffler functions}\label{MittagLefflerFRAC}

For a real-valued function $g:[0,\infty)\rightarrow \mathbb{R}$ satisfying  $\displaystyle \sup_{t \in [0,\infty)} e^{-\omega t}| g(t)| <\infty$, the Laplace transform of $g$ is defined as
\begin{eqnarray}
	\label{LT} G(s):=\mathcal{L} [g(t)](s)=
	\int_{0}^{\infty} e^{-st} g(t) \, \, dt &  (\Re(s)>\omega).
\end{eqnarray}

Here, we notice that the function $G$ (the Laplace image of $g$) is analytic on the right-half plane $\displaystyle \left\{s\in \BC~:~\Re(s)>\omega\right\},$ whereby $\omega$ is chosen as the infimum of the values of $s$ for which the right-hand side of (\ref{LT}) is convergent. Further details may be found e.g. in \cite[Subsection 1.2]{KST06}.

For the inverse of the Laplace transform, defined as $g(t)=\mathcal{L}^{-1} [G(s)](t)$, the integration is performed along the strip $[c-i\infty,c+i\infty]$ ($c:=\Re(s)>\omega$), and whence, the resulting formula is independent of the choice of $c$. 

Associated to the Laplace transform is the {\it Laplace convolution formula} (cf.~\cite[(1.4.10), (1.4.12) of p. 19]{KST06} and \cite[(2.237), (2.238) of pp. 103--104]{Podlubny99}), defined for two functions $g$ and $h$ by 
\begin{eqnarray}
	\label{LaplaceConvolution} \mathcal{L}\left[\int_0^t h(t-\tau)g(\tau ) d\tau \right](s)=\mathcal{L}[h](s)\cdot\mathcal{L}[g](s).
\end{eqnarray}

Of foremost importance are also the following operational identities, involving the Caputo-Djrbashian derivative (\ref{CaputoDerivative}) and the Riemann-Liouville integral (\ref{RiemannLiouville}) respectively. Namely, the property
\begin{eqnarray}
	\label{LaplaceCaputo} \mathcal{L}[\partial_t^\gamma g(t)]=s^{\gamma}\mathcal{L}[g](s)-\sum_{k=0}^{m-1}s^{\gamma-k-1}\partial_t^k g(0) & \left(~m-1<\gamma\leq m~\right)
\end{eqnarray}
is fulfilled whenever $\mathcal{L}[g(t)](s)$, $\mathcal{L}[\partial_t^m g(t)](s)$ exist and
\begin{eqnarray*}
	\displaystyle \lim_{t \rightarrow \infty}\left(\partial_t\right)^k g(t)=0 &\left(~k=0,1,\ldots,m-1~\right),
\end{eqnarray*}
holds for every $g\in C^m(0,\infty)$ such that $\left(\partial_t\right)^mg\in L^1(0,b)$ for any $b>0$ (cf.~\cite[Lemma 2.24]{KST06}).
On the other hand, the identity
\begin{eqnarray}
	\label{LaplaceRLiouville}  \mathcal{L}[I_{0^+}^{\gamma} g(t)]=s^{-\gamma}\mathcal{L}[g](s) & (\Re(s)>\omega)
\end{eqnarray}
is always satisfied in case of $g\in L^1(0,b)$, for any $b>0$ (cf.~\cite[Lemma 2.14]{KST06}).

Next, we turn our attention to some of the special functions to be considered on the sequel. 
We introduce the two-parameter/three-parameter Mittag-Leffler functions $E_{\alpha,\beta}$ resp. $E_{\alpha,\beta}^\gamma$, as power series expansions of the form
\begin{eqnarray}
	\label{MittagLefflerTwo}E_{\alpha,\beta}(z)=\displaystyle \sum_{k=0}^{\infty} \frac{z^k}{\Gamma(\alpha k+\beta)} 
	&	(\Re(\alpha)>0),
	\\
	\label{MittagLefflerThree}E_{\alpha,\beta}^\gamma(z)=\displaystyle \sum_{k=0}^{\infty} \frac{(\gamma)_k}{k!}\frac{z^k}{\Gamma(\alpha k+\beta)}
	& (\Re(\alpha)>0),
\end{eqnarray}
where $\Gamma(\cdot)$ stands for the Euler's Gamma function
\begin{eqnarray}
	\label{GammaFunction} \Gamma(z)=\int_0^\infty e^{-t} t^{z-1}dt &(\Re(z)>0),
\end{eqnarray}
and $\displaystyle (\gamma)_k=\frac{\Gamma(\gamma+k)}{\Gamma(\gamma)}$ ($\Re(\gamma)>-k$) for the Pochhammer symbol.

The wide class of Mittag-Leffler functions, defined viz (\ref{MittagLefflerTwo}) and (\ref{MittagLefflerThree}), permits us to represent several transcendental and special functions (cf.~\cite[Subsection 1.8 \& Subsection 1.9]{KST06} and \cite[Chapter 3 -Chapter 5]{GKMaiR14}), such as:
\begin{itemize}
	\item[\bf (i)] the exponential function $$e^z=E_{1,1}(z);$$
	\item[\bf (ii)] the error function $$\displaystyle \mbox{erf}(z):=\frac{2}{\sqrt{\pi}}\int_0^z e^{-t^2}dt=1-e^{-z^2}E_{\frac{1}{2},1}(-z);$$
	\item[\bf (iii)] the hyperbolic functions
	\begin{eqnarray*}
		\cosh(z)=E_{2,1}(z^2) &\mbox{and} &\sinh(z)=zE_{2,2}(z^2);
	\end{eqnarray*}
	\item[\bf (iv)] the complex exponential function $$e^{iz}:=\cos(z)+i\sin(z)=E_{2,1}(-z^2)+izE_{2,2}(-z^2);$$
	\item[\bf (v)] the Kummer confluent hypergeometric function
	$$ \varPhi(\gamma;\beta;z):={~}_1F_1(\gamma;\beta,z) =\Gamma(\beta)E_{1,\beta}^\gamma(z);$$
	\item[\bf (vi)] the generalized hypergeometric function
	\begin{eqnarray*}
		{~}_1F_m\left(\gamma;\frac{\beta}{m},\frac{\beta+1}{m},\ldots,\frac{\beta+m-1}{m};\frac{z^m}{m^m}\right) =\Gamma(\beta)E_{m,\beta}^\gamma(z) &(m\in \BN).
	\end{eqnarray*}
\end{itemize}

Here, we would like to stress that $E_{\alpha,\beta}^\gamma$ corresponds to a faithful generalization of $E_{\alpha,\beta}$. Indeed, from the Pochhammer identity $(1)_k=k!$ it immediately follows that $E_{\alpha,\beta}^1$ ($\gamma=1$) coincides with $E_{\alpha,\beta}$.
For these functions, we have the following Laplace transform identity (cf.~\cite[eq. (5.1.26) of p. 102]{GKMaiR14} \& \cite[eq.~(1.9.13) of p.~47]{KST06})
\begin{eqnarray}
	\label{LaplaceIdentityMittagLeffler}
	\mathcal{L}[t^{\beta-1}E_{\alpha,\beta}^\gamma(\lambda t^\alpha)](s)=\frac{s^{\alpha \gamma-\beta}}{(s^\alpha-\lambda)^\gamma} & (\Re(s)>0;~\Re(\beta)>0;~|\lambda s^{-\alpha}|<1).
\end{eqnarray}

In the remainder part of this subsection we list some of the properties required for the sake of the rest of the paper. 
We start to point out the following decay estimates, involving the two-parameter Mittag-Leffler functions (\ref{MittagLefflerTwo}), that will be employed on the sequel of results to be proved in Section \ref{DispersiveSection}.
\begin{theorem}[cf.~\cite{Podlubny99}, \textsc{Theorem 1.6}]\label{EstimatesMittagLeffler}
	If $\alpha\leq 2$, $\beta$ is an arbitrary real number, $\theta$ is such that $\frac{\pi\alpha}{2}<\theta<\min\{\pi,\pi\alpha\}$ and $C$ is a real constant, then
	\begin{eqnarray*}
		|E_{\alpha,\beta}(z)|\leq \frac{C}{1+|z|} &(z \in\mathbb{C}~~;~~\theta\leq |\arg(z)|\leq \pi).
	\end{eqnarray*} 
\end{theorem}

\begin{theorem}[cf.~\cite{Simon14}, {\bf Theorem 4.} \& \cite{BS21}, {\bf Proposition 4.}]\label{OptimalEstimatesMittagLeffler}
	The following optimal estimates are fulfilled by $E_{\alpha,\beta}(-z)$:
	\begin{eqnarray*}
		\dfrac{1}{1+\Gamma(1-\alpha)z} \leq E_{\alpha,1}(-z)\leq \dfrac{1}{1+\frac{1}{\Gamma(1+\alpha)}z} ~~ (z\geq0~;~0< \alpha\leq 1);\\
		\dfrac{1}{\left(1+\sqrt{\frac{\Gamma(1-\alpha)}{\Gamma(1+\alpha)}}z\right)^2}\leq \Gamma(\alpha) E_{\alpha,\alpha}(-z)\leq \dfrac{1}{\left(1+\sqrt{\frac{\Gamma(1+\alpha)}{\Gamma(1+2\alpha)}}z\right)^2} ~~ (z\geq0~;~0< \alpha\leq 1); \\
		\dfrac{1}{1+\frac{\Gamma(\beta-\alpha)}{\Gamma(\beta)}z} \leq \Gamma(\beta)E_{\alpha,\beta}(-z)\leq \dfrac{1}{1+\frac{\Gamma(\beta)}{\Gamma(\beta+\alpha)}z} ~~ (z\geq0;0<\alpha\leq 1;\beta>\alpha).
	\end{eqnarray*}
\end{theorem}

We finish this subsection with the next lemma that will be applied, in Section \ref{MainSection}, on the proof of the solution representation for both Cauchy problems, {\bf Cauchy Problem \ref{CP_parabolic}} resp. {\bf Cauchy Problem \ref{CP_hyperbolic}}. 
\newpage 
\begin{lemma}[cf.~\cite{GKMaiR14}, p. 99]
	\label{Properties_ML}
	Let $z\in \mathbb{C}$ be given.
	\begin{enumerate}
		\item[\bf (i)]
		If $\alpha,\beta,\gamma \in \BC$ are such that $\Re(\alpha)>0$, $\Re(\beta)>0$, $\Re(\beta-\alpha)>0$, then
		$$
		\displaystyle{zE_{\alpha,\beta}^\gamma(z)=E_{\alpha,\beta-\alpha}^\gamma(z)-E_{\alpha,\beta-\alpha}^{\gamma-1}(z)}.
		$$
		\item[\bf (ii)] If $\alpha,\beta,\gamma \in \BC$ are such that $\Re(\alpha)>0$, $\Re(\beta)>0$, $\alpha-\beta\notin \BN_0$, then
		$$
		\displaystyle{zE_{\alpha,\beta}(z)=E_{\alpha,\beta-\alpha}(z)-\frac{1}{\Gamma(\beta-\alpha)}}.
		$$
		\item[\bf (iii)] If $\alpha,\beta \in \BC$ are such that $\Re(\alpha)>0$, $\Re(\beta)>1$, then
		$$
		\displaystyle{ \alpha E_{\alpha,\beta}^2(z)=E_{\alpha,\beta-1}(z)-(1+\alpha-\beta)E_{\alpha,\beta}(z)}.
		$$
	\end{enumerate}
\end{lemma}

\subsection{Fourier analysis, function spaces and associated operators}\label{FourierAnalysisSetup}

Let us now denote by $\mathcal{S}(\BR^n)$ the Schwartz space over $\BR^n$ and by $\mathcal{S}'(\BR^n)$ the \textit{space of tempered distributions} (dual space of $\mathcal{S}(\BR^n)$). 
Throughout this paper we also denote by $L^p$ ($1\leq p\leq \infty$) the standard Lebesgue spaces and by $\|\cdot\|_p$ its norm. We adopt the bracket notation $\langle \cdot,\cdot \rangle$ as the inner product underlying to the Hilbert space $L^2$.
The Fourier transform of $\varphi\in \mathcal{S}(\BR^n)$ is defined as (cf.~\cite[Chapter A]{Tao06})
\begin{equation}
	\label{FT} \widehat{\varphi}(\xi):=(\mathcal{F} \varphi)(\xi)=
	\int_{\BR^n} \varphi(x)e^{-ix \cdot \xi}  \, \, dx
\end{equation}
and the corresponding {\it inverse Fourier transform} as
\begin{equation}
	\label{FInv} \varphi(x):=(\mathcal{F}^{-1}\widehat{\varphi})(x)=\frac{1}{(2\pi)^{n}}
	\int_{\BR^n} \widehat{\varphi}(\xi)e^{ix \cdot \xi}  \, \, d\xi \, ,
\end{equation}
where $x\cdot \xi $ denotes the standard Euclidean inner product between $x,\xi\in \BR^n$.

We note that the action of the isomorphism $\mathcal{F}:\mathcal{S}(\BR^n)\longrightarrow \mathcal{S}(\BR^n)$ can be extended to $L^p-$spaces, using the fact that $\mathcal{S}(\BR^n)$ is a dense subspace of $L^p(\BR^n)$, for values of $1\leq p<\infty$, and to $\mathcal{S}'(\BR^n)$ via the duality relation
\begin{eqnarray*}
	\langle \widehat{\varphi},\psi \rangle=\langle \varphi,\widehat{\psi}\rangle, & \varphi\in \mathcal{S}(\BR^n), & \psi\in \mathcal{S}'(\BR^n).
\end{eqnarray*}

In particular, {\it the Plancherel identity}
\begin{eqnarray}
	\label{PlancherelId}
	\displaystyle  \frac{1}{{(2\pi)^{\frac{n}{2}}}}\|  \widehat{\varphi}\|_2 =\| \varphi\|_2
\end{eqnarray}
follows from the fact that the automorphism $\displaystyle \varphi\mapsto \frac{1}{{(2\pi)^{\frac{n}{2}}}}\widehat{\varphi}$ underlying to $\mathcal{S}(\BR^n)$ induces an unitary operator over $L^2(\BR^n)$.

The Fourier transform $(\ref{FT})$ and the convolution product over $\BR^n$:
\begin{eqnarray}
	\label{ConvolutionRn}(\varphi \ast \psi)(x)=\int_{\BR^n} \varphi(x-y) \psi(y)dy,
\end{eqnarray}
intertwined by the {\it Fourier convolution formula} (cf.~\cite[Theorem 5.8]{LiebLoss01})
\begin{eqnarray}
	\label{FourierConvolutionRn} \mathcal{F}(\varphi \ast \psi)(\xi)=\widehat{\varphi}(\xi) \widehat{\psi}(\xi)
\end{eqnarray}
share many interesting features of the $L^p-$spaces. Of special interest so far
is the Young's inequality (cf.~\cite[Theorem 4.2]{LiebLoss01})
\begin{eqnarray}
	\label{YoungIneq}\| \psi \ast \varphi\|_q \leq \left(\frac{\kappa_p\kappa_{r'}}{\kappa_q}\right)^n\|\psi \|_{r'} \| \varphi\|_p,
\end{eqnarray}
that holds for every $1\leq p,q,r'\leq \infty~$ such  that $\displaystyle ~\frac{1}{p}+\frac{1}{r'}=\frac{1}{q}+1.$
Here and elsewhere, $\displaystyle \kappa_s=\sqrt{{s^{\frac{1}{s}}}{(s')^{-\frac{1}{s'}}}}$ ($\displaystyle \frac{1}{s}+\frac{1}{s'}=1$) stands for the {\it sharp} constant appearing on the right-hand side of (\ref{YoungIneq}).

Despite the Caputo-Djrbashian derivative (\ref{CaputoDerivative}), the fractional differential operator $(-\Delta)^{\frac{\gamma}{2}}$ ($\gamma\geq 0$), defined viz the spectral formula
\begin{eqnarray}
	\label{RieszOperators}
	\mathcal{F}\left((-\Delta)^{\frac{\gamma}{2}}\varphi\right)(\xi)=|\xi|^\gamma \widehat{\varphi}(\xi),
\end{eqnarray}
or equivalenty, as $(-\Delta)^{\frac{\gamma}{2}}=\mathcal{F}^{-1}|\xi|^{\gamma}\mathcal{F}$,
is also required on the formulation of {\bf Cauchy Problem \ref{CP_parabolic}} \& {\bf Cauchy Problem \ref{CP_hyperbolic}}. Underlying to $(-\Delta)^{\frac{\gamma}{2}}$ are the 
{\it homogeneous Sobolev spaces} of order $\gamma\geq 0$,
\begin{eqnarray*}
	\label{SobolevHomogeneous}	\dot{W}^{\gamma,p}(\BR^n):=\left\{~\varphi\in \mathcal{S}'(\BR^n)~:~(-\Delta)^{\frac{\gamma}{2}}\varphi\in L^p(\BR^n)~\right\},
\end{eqnarray*}
induced by the seminorm $\displaystyle\varphi\mapsto \left\|~(-\Delta)^{\frac{\gamma}{2}}\varphi~\right\|_{p}$.

From now on, we will adopt mostly the $L^p-$notation when we are referring to the {\it homogeneous Sobolev space} of order zero (0), $\displaystyle \dot{W}^{0,p}(\BR^n)={L}^p(\BR^n)$. 

Also, the so-called {\it Riesz fractional integral operator} of order $\gamma$, defined for values of $0<\gamma<n$ by (cf.~\cite[Section 25]{samko1993fractional})
\begin{eqnarray}
	\label{RieszIntegral}(-\Delta)^{-\frac{\gamma}{2}}\varphi(x)=(R_\gamma\ast\varphi)(x),~\mbox{with}&\displaystyle R_\gamma(x)=2^{-\gamma}\pi^{-\frac{n}{2}}\frac{\Gamma\left(\frac{n-\gamma}{2}\right)}{\Gamma\left(\frac{\gamma}{2}\right)}\frac{1}{|x|^{n-\gamma}}
\end{eqnarray}
can also be represented in terms of its Fourier multiplier. Namely, there is (cf.~\cite[Theorem 25.1]{samko1993fractional})
\begin{eqnarray*}
	\mathcal{F}\left((-\Delta)^{-\frac{\gamma}{2}}\varphi\right)(\xi)=|\xi|^{-\gamma} \widehat{\varphi}(\xi), & \mbox{for} & 0<\gamma<n.
\end{eqnarray*}

On the sequel of main results to be proved in the end of Section \ref{MainSection}, it will also be necessary the notion of {\it mixed-normed Lebesgue spaces} $L^s_tL^q_x$. For values of $1\leq s<\infty$, we define $L^s_tL^q_x([0,T]\times \BR^n)$, with $0<T<\infty$, as the Banach space with norm
\begin{eqnarray}
	\label{mixedNorm}\| u\|_{L^s_tL^q_x([0,T]\times \BR^n)}=\left(\int_0^T \|u(t,\cdot)\|_q^s dt\right)^{\frac{1}{s}}& (1\leq q\leq \infty).
\end{eqnarray}

For $s=\infty$, the space $L^\infty_t L^q_x([0,T]\times \BR^n)$ is defined in terms of the norm
\begin{eqnarray}
	\label{mixedNormSup}\| u\|_{L^\infty_t L^q_x([0,T]\times \BR^n)}=\mbox{ess}\sup_{t\in [0,T]}\|u(t,\cdot)\|_q & (1\leq q\leq \infty).
\end{eqnarray}

\section{Dispersive Estimates}\label{DispersiveSection}

\subsection{Proof Strategy}\label{ProofStrategy}

In this section we shall deduce estimates for pseudo-differential operators of the type $\displaystyle (-\Delta)^{\frac{\eta}{2}}E_{\alpha,\beta}\left(~-\lambda(-\Delta)^{\frac{\sigma}{2}} t^\alpha~\right)$, represented through the pseudo-differential formula
\begin{eqnarray}
	\label{MittagLefflerOperator}
	\displaystyle	(-\Delta)^{\frac{\eta}{2}}E_{\alpha,\beta}\left(~-\lambda(-\Delta)^{\frac{\sigma}{2}} t^\alpha~\right)\varphi(x)=  \nonumber \\
	=\frac{1}{(2\pi)^n}\int_{\mathbb{R}^n}|\xi|^\eta E_{\alpha,\beta}(-\lambda|\xi|^\sigma t^\alpha)\widehat{\varphi}(\xi)e^{ix \cdot \xi}d\xi,
\end{eqnarray}
where $0< \alpha\leq 1$, $\beta>0$, $\sigma>0$, $\lambda \in \mathbb{C}$ and $\eta\in \mathbb{R}$.

Recall that the {\it Fourier convolution formula} (\ref{FourierConvolutionRn}) together with eq.~(\ref{ConvolutionRn}) states that eq. (\ref{MittagLefflerOperator}) can be rewritten as
\begin{eqnarray}
	\label{MittagLefflerConvolution}
	(-\Delta)^{\frac{\eta}{2}}E_{\alpha,\beta}\left(~-\lambda(-\Delta)^{\frac{\sigma}{2}} t^\alpha~\right)\varphi(x)=\displaystyle \left(~{\bm K}_{\sigma,n}^{\eta}(t,\cdot|\alpha,\beta,\lambda)\ast \varphi~\right)(x),
\end{eqnarray}
with \begin{eqnarray}
	\label{MittagLefflerKernel}
	{\bm K}_{\sigma,n}^{\eta}(t,x|\alpha,\beta,\lambda)=
	\frac{1}{(2\pi)^n}\int_{\mathbb{R}^n}|\xi|^\eta E_{\alpha,\beta}(-\lambda|\xi|^\sigma t^\alpha)e^{ix \cdot \xi}d\xi.
\end{eqnarray}

Then, for every $1\leq p,q,r' \leq \infty$ satisfying $\displaystyle \frac{1}{p}+\frac{1}{r'}=\frac{1}{q}+1$, the condition $E_{\alpha,\beta}\left(~-\lambda(-\Delta)^{\frac{\sigma}{2}} t^\alpha~\right)\varphi\in \dot{W}^{\eta,q}(\BR^n)$ is satisfied for every $\varphi\in L^p(\BR^n)$, whenever ${\bm K}_{\sigma,n}^{\eta}(t,\cdot|\alpha,\beta,\lambda)\in L^{r'}(\BR^n)$.

On the other hand, due to the fact that $\displaystyle \xi \mapsto |\xi|^\eta E_{\alpha,\beta}\left(-\lambda|\xi|^\sigma t^\alpha\right)$ is radially symmetric one can recast eq. (\ref{MittagLefflerKernel}) in terms of the Hankel transform of order $\nu$, $\widetilde{\mathcal{H}}_\nu$ (cf. \cite{DeCarli08}), defined for values of $\nu>-1$ by
\begin{eqnarray*}
	\widetilde{\mathcal{H}}_{\nu}\phi(\tau)=\int_{0}^\infty \phi(\rho)~(\tau\rho)^{-\nu}J_{\nu}(\tau\rho)\rho^{2\nu+1}d\rho & (\tau>0),
\end{eqnarray*}
where $J_\nu$ stands for the Bessel function of order $\nu$ (cf.~\cite[Subsection 1.7]{KST06}).

Namely, from the Fourier inversion formula for radially symmetric functions (cf. \cite[p. 485, Lemma 25.1]{samko1993fractional})
\begin{eqnarray*}
	\label{FourierInversionRadial}
	\int_{\BR^n} \phi(|\xi|)e^{i x\cdot \xi} d\xi =\frac{(2\pi)^{\frac{n}{2}}}{|x|^{\frac{n}{2}-1}} \int_{0}^{\infty} \phi(\rho) J_{\frac{n}{2}-1}(\rho|x|)~\rho^{\frac{n}{2}}d\rho
\end{eqnarray*}
it follows immediately that eq. (\ref{MittagLefflerKernel}) simplifies to
\begin{eqnarray}
	\label{MittagLefflerHankel}
	{\bm K}_{\sigma,n}^\eta(t,x|\alpha,\beta,\lambda)&=&\displaystyle \frac{1}{(2\pi)^{\frac{n}{2}}|x|^{\frac{n}{2}-1}} \int_{0}^{\infty} \rho^\eta E_{\alpha,\beta}(-\lambda\rho^\sigma t^\alpha) J_{\frac{n}{2}-1}(\rho|x|)~\rho^{\frac{n}{2}}d\rho \nonumber \\
	&=&\displaystyle \frac{1}{(2\pi)^{\frac{n}{2}}} \widetilde{\mathcal{H}}_{\frac{n}{2}-1}\left[~|x|^\eta E_{\alpha,\beta}(-\lambda|x|^\sigma t^\alpha)~\right].
\end{eqnarray}

Thereby, the representation formula (\ref{MittagLefflerHankel}) shows in turn that the condition ${\bm K}_{\sigma,n}^{\eta}(t,\cdot|\alpha,\beta,\lambda)\in L^{r'}(\BR^n)$ is assured by the boundedness of the Hankel transform $$\widetilde{\mathcal{H}}_{\frac{n}{2}-1}:L^r((0,\infty),\rho^{n-1}d\rho)\rightarrow L^{r'}((0,\infty),\rho^{n-1}d\rho),$$ for values of $1\leq r\leq 2\leq r'\leq \infty$ satisfying $\displaystyle \frac{1}{r}+\frac{1}{r'}=1$ -- also known as Fourier-Bessel transform (see e.g.~\cite{colzani1993equiconvergence} and the references therein). 

It is worth mentioning that the proof that $\widetilde{\mathcal{H}}_{\frac{n}{2}-1}$ is an isometric operator on $L^{2}((0,\infty),\rho^{n-1}d\rho)$ ($r=r'=2$) can be easily reached by the eq.~(\ref{MittagLefflerHankel}) and the {\it Plancherel identity} (\ref{PlancherelId}), highlighted on Subsection \ref{FourierAnalysisSetup}.
On the other hand, for $r=1$ and $r'=\infty$ the set of identities (cf.~\cite[eq.~(1.7.4)]{KST06}):
\begin{eqnarray*}
	J_{-\frac{1}{2}}(z)=\left(\frac{2}{\pi z}\right)^{\frac{1}{2}}\cos(z) &\mbox{and} & J_{\frac{1}{2}}(z)=\left(\frac{2}{\pi z}\right)^{\frac{1}{2}}\sin(z),
\end{eqnarray*}
together with the {\it Poisson integral representation} of $J_\nu$ (cf.~\cite[eq.~(1.7.5)]{KST06}):
\begin{eqnarray*}
	J_\nu(z)=\frac{\left(\frac{z}{2}\right)^\nu}{\sqrt{\pi}\Gamma(\nu+1)}\int_{-1}^1 (1-\rho^2)^{\nu-\frac{1}{2}}\cos(z\rho) d\rho& \left(~\Re(\nu)>-\frac{1}{2}~\right),
\end{eqnarray*}
shows in turn that the {\it supremum norm} $\displaystyle b_{\frac{n}{2}-1}:=\sup_{z\in (0,\infty)}\left|z^{1-\frac{n}{2}}J_{\frac{n}{2}-1}(z)\right|$ is bounded above, and hence
\begin{eqnarray*}
	\left|\widetilde{\mathcal{H}}_{\frac{n}{2}-1}\phi(\tau)\right|\leq b_{\frac{n}{2}-1} \|\phi\|_{L^1((0,\infty),\rho^{n-1}d\rho)}.
\end{eqnarray*}

Moreover, with the aid of the celebrated {\it Riesz-Thorin convexity theorem} we are able to obtain $L^r-L^{r'}$estimates, for values of $1\leq r\leq 2\leq r'\leq \infty$ satisfying $\displaystyle \frac{1}{r}+\frac{1}{r'}=1$, as an interpolation between $L^1-L^\infty$ and $L^2-L^2$ estimates. Namely, for $\displaystyle \frac{2}{r'}=2-\frac{2}{r}$, it can be shown that the norm of $\widetilde{\mathcal{H}}_{\frac{n}{2}-1}$ is bounded above by $\displaystyle \left(~b_{\frac{n}{2}-1}~\right)^{1-\frac{2}{r'}}$ (cf.~\cite[eq.~(2.5)]{DeCarli08}).

We would like to stress that the norm $\displaystyle \widetilde{h}_{\frac{n}{2}-1}(r,r'):=\left\|~\widetilde{\mathcal{H}}_{\frac{n}{2}-1}~\right\|_{L^r\rightarrow L^{r'}}$ cannot be computed accurately for $(r,r')=(1,\infty)$. Nevertheless, using the groundbreaking theorems of Beckner and Lieb (cf.~\cite{beckner1975inequalities,lieb1990gaussian}), it can be computed with precision in case where $1<r\leq 2$ as it was proved on De Carli's paper \cite{DeCarli08} (cf. \cite[Proposition 2.1.]{DeCarli08}).

In brief, with the aid of the mapping properties of the Hankel transform we are able to achieve wisely the underlying endpoint $L^{r'}-$estimates for the kernel function 
(\ref{MittagLefflerKernel}) from the endpoint $L^{r}-$estimates for the Fourier multiplier $|\xi|^\eta E_{\alpha,\beta}(-\lambda|\xi|^\sigma t^\alpha)$ of (\ref{MittagLefflerOperator}). 
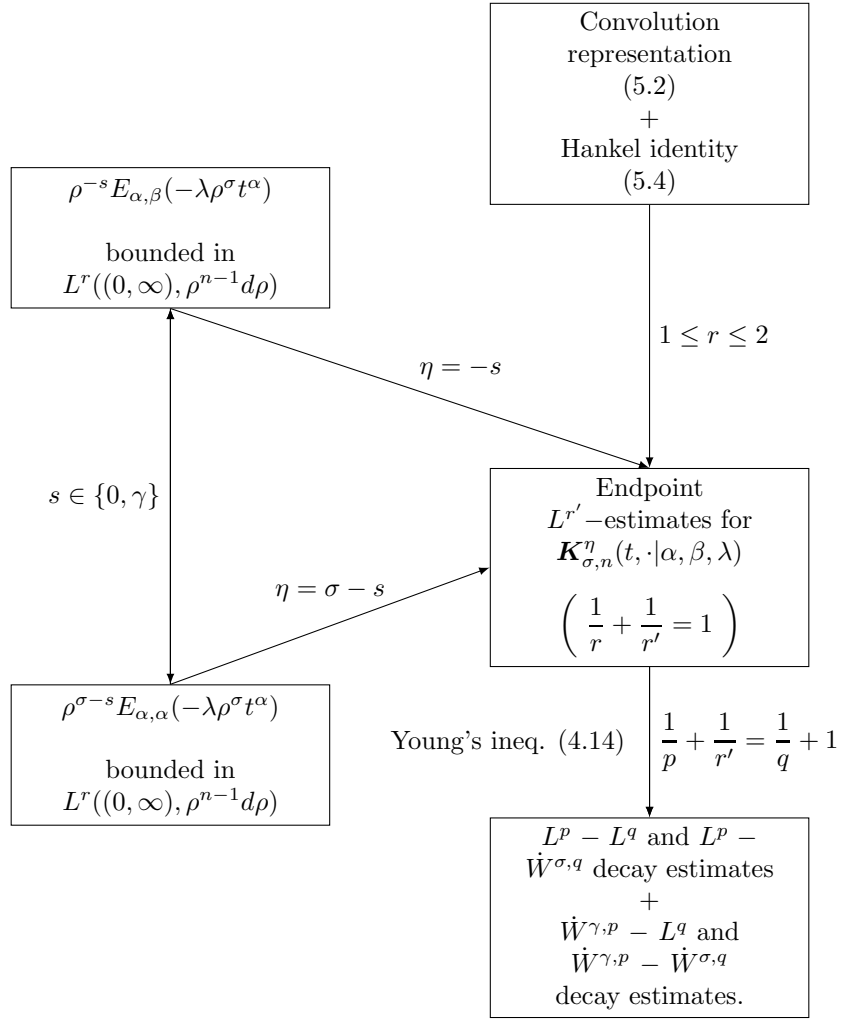
\begin{figure}\label{DispersivePicture}
	\begin{tikzpicture}
		\node[mynode] (v1){
			{$ \rho^{-s}E_{\alpha,\beta}(-\lambda\rho^\sigma t^\alpha)$} \\ \ \\ bounded in \\ {$L^r((0,\infty),\rho^{n-1}d\rho)$}};
		\node[mynode,below right=3.0cm of v1](v2) {
			{Endpoint $L^{r'}-${estimates for} }  \\ 
			{$\displaystyle{\bm K}_{\sigma,n}^{\eta}(t,\cdot|\alpha,\beta,\lambda)$}\\ \ \\{$\displaystyle \left(~\frac{1}{r}+\frac{1}{r'}=1~\right)$}};
		
		\node[mynode,below = 5.0cm of v1] (v4){
			{$\rho^{\sigma-s} E_{\alpha,\alpha}(-\lambda\rho^\sigma t^\alpha)$} \\ \ \\ bounded in \\ {$L^r((0,\infty),\rho^{n-1}d\rho)$}};
		\node[mynode,below=2.0cm of v2](v3) {
			{{$L^p-L^q$} and {$L^p-\dot{W}^{\sigma,q}$} decay estimates} \\ + \\ 
			{ {$\dot{W}^{\gamma,p}-L^q$} and  {$\dot{W}^{\gamma,p}-\dot{W}^{\sigma,q}$}\\  decay estimates.}};
		\node[mynode, above=3.5cm of v2](v5) {{Convolution representation } \\
			{(\ref{MittagLefflerConvolution})}\\ + \\	{Hankel identity} \\
			{(\ref{MittagLefflerHankel})} };
		\draw[-latex] (v1.south) -- node[auto,] {} (v4.north);
		\draw[-latex] (v4.north) -- node[right=0.2mm] {~} (v1.south);
		\draw[-latex] (v4.north) -- node[left=0.2mm] {{$s\in \{0,\gamma\}$}} (v1.south);
		\draw[-latex] (v1.south) -- node[auto,] {{$\eta=-s$}} (v2.north);
		\draw[-latex] (v4.north) -- node[above=2mm] {{$\eta=\sigma-s$}} (v2.west);
		\draw[-latex] (v5.south) -- node[auto,]{{$1\leq r\leq 2$}} (v2.north);
		\draw[-latex] (v2.south) -- node[left=2mm, align=center]{Young's ineq. (\ref{YoungIneq})} (v3.north);
		\draw[-latex] (v2.south) -- node[right=0.1mm, align=center]{{$\displaystyle \frac{1}{p}+\frac{1}{r'}=\frac{1}{q}+1$}} (v3.north);
	\end{tikzpicture}
	\caption{\textit{Tour de force} of the decay estimates to be obtained on {\bf Subsection \ref{LpLqEstimates}} and {\bf Subsection \ref{SobolevEstimates}}.}
\end{figure}

On the sequence of results to be proved on the subsequent subsections and elsewhere, the constants to be adopted are listed below to avoid cluttering up the notations.
\begin{notation}\label{SharpConstants}
	We denote by
	\begin{itemize}
		\item[\bf (a)] $\displaystyle \omega_{n-1}=\frac{2\pi^{\frac{n}{2}}}{\Gamma\left(\frac{n}{2}\right)}$ the $(n-1)-$dimensional measure of the sphere $\mathbb{S}^{n-1}$; 
		\item[\bf (b)] $\displaystyle \kappa_s=\sqrt{{s^{\frac{1}{s}}}{(s')^{-\frac{1}{s'}}}}$ (with $\displaystyle \frac{1}{s}+\frac{1}{s'}=1$) the sharp constant underlying to Young's inequality (\ref{YoungIneq});
		\item[\bf (c)] $\displaystyle \widetilde{h}_{\frac{n}{2}-1}(r,r'):=\left\|~\widetilde{\mathcal{H}}_{\frac{n}{2}-1}~\right\|_{L^r\rightarrow L^{r'}}$ (with $1\leq r\leq 2$ and $\displaystyle \frac{1}{r}+\frac{1}{r'}=1$) the operator norm of $$\widetilde{\mathcal{H}}_{\frac{n}{2}-1}:L^r((0,\infty),\rho^{n-1}d\rho)\rightarrow L^{r'}((0,\infty),\rho^{n-1}d\rho);$$
		\item[\bf (d)] 
		$C_{r,\sigma,n}^{(s)}(\alpha,\beta,\lambda)$ the constant obtained Lemma \ref{LrMittagLeffler};
		\item[\bf (e)] 
		$D_{r,\sigma,n}^{(s)}(\alpha,\lambda)$ the constant obtained in Lemma \ref{LrMittagLeffler2}.
	\end{itemize}
\end{notation}

\subsection{$L^p-L^q$ and $L^p-\dot{W}^{\sigma,q}$ decay estimates}\label{LpLqEstimates}
Our first aim is to establish under which conditions $E_{\alpha,\beta}\left(~-\lambda(-\Delta)^{\frac{\sigma}{2}} t^\alpha~\right)\varphi$ belongs to $L^q(\BR^n)$ or to $\dot{W}^{\sigma,q}(\BR^n)$.
In concrete, the following Mellin's integral identities
\begin{eqnarray}
	\label{MellinId}
	\displaystyle \int_{0}^\infty \frac{\rho^{n-rs-1}}{(1+b\rho^\sigma)^a}d\rho = \frac{\Gamma\left(\frac{ n-rs}{\sigma}\right)\Gamma\left(\frac{a\sigma+rs-n}{\sigma}\right)}{\sigma~\Gamma(a)}~b^{-\frac{n-rs}{\sigma}}\\ ~\displaystyle \left(~0<\Re\left(\frac{n-rs}{\sigma}\right)<\Re(a)\right),~\nonumber
	\\ \nonumber \\
	\label{MellinId2} 		\displaystyle \int_{0}^\infty \frac{\rho^{ n+r(\sigma-s)-1}}{(1+b\rho^\sigma)^{2r}}d\rho
	= \frac{\Gamma\left(\frac{n+r(\sigma-s)
		}{\sigma}\right)\Gamma\left(\frac{r(\sigma+s)-n}{\sigma}\right)}{\sigma~\Gamma(2r)}~b^{-\frac{n+r(\sigma-s)}{\sigma}}\\ \displaystyle \left(~0<\Re\left(\frac{n-rs}{\sigma}\right)<\Re(r)~\right),\nonumber
\end{eqnarray}
that yield from \cite[eq.~(1.4.65) of p.~24]{KST06}, together with Theorem \ref{EstimatesMittagLeffler} and Theorem \ref{OptimalEstimatesMittagLeffler} will be used on the proof of Lemma \ref{LrMittagLeffler} and Lemma \ref{LrMittagLeffler2} to ensure afterwards, on the proof of Proposition \ref{YoungIneqProposition}, that the kernel ${\bm K}_{\sigma,n}^\eta(\cdot,x|\alpha,\beta,\lambda)$, represented through eq. (\ref{MittagLefflerHankel}), belongs to $L^{r'}(\BR^n)$. The proof of both results is {enclosed on {\bf \ref{AppendixA}} and {\bf \ref{AppendixB}}}.

\begin{lemma}[see Appendix \ref{AppendixA}]\label{LrMittagLeffler}  	Let $0<\alpha\leq 1$ and $\beta>0$ be given. We assume that $\sigma>0$, $r\geq 1$ and $\displaystyle s\in \mathbb{R}$ satisfy one of the following conditions:
	\begin{itemize}
		\item[{\bf (i)}] $r=1$~~$\wedge$~~ $n-\sigma<s<n$;
		\item[{\bf (ii)}] $\displaystyle \max\left\{1,\frac{n}{\sigma}\right\} <r<\infty$~~$\wedge$~~$s=0$;  
		\item[{\bf (iii)}] $\displaystyle \max\left\{1,\frac{n}{\sigma+s}\right\} <r<\frac{n}{s}$~~$\wedge$~~$0<s<n$. 
	\end{itemize}	Then, one has	\begin{eqnarray}
		\label{IneqC}
		\displaystyle  \left(\int_0^\infty\left|~ \rho^{-s}E_{\alpha,\beta}(-\lambda \rho^\sigma t^{\alpha})~\right|^r \rho^{n-1}d\rho\right)^{\frac{1}{r}}
		\leq  \nonumber \\ \frac{\left(~C_{r,\sigma,n}^{(s)}(\alpha,\beta,\lambda) \right)^{\frac{1}{r}}}{\left|~\Gamma(\beta)~\right|}~t^{-\frac{\alpha}{\sigma}\left(\frac{n}{r}-s\right)},
	\end{eqnarray}
	with $C_{r,\sigma,n}^{(s)}(\alpha,\beta,\lambda)$ equals to
	\begin{eqnarray*}
		\begin{cases} \displaystyle \frac{\Gamma\left(\frac{ n-rs}{\sigma}\right)\Gamma\left(\frac{r(2\sigma+s)-n}{\sigma}\right)}{\sigma~\Gamma(2r)}\left(\sqrt{\frac{\Gamma(1+\alpha)}{\Gamma(1+2\alpha)}}~\lambda  \right)^{-\frac{n-rs}{\sigma}} \displaystyle~~~~,~\lambda\geq 0,\beta=\alpha
			\\ \ \\
			\displaystyle \frac{\Gamma\left(\frac{ n-rs}{\sigma}\right)\Gamma\left(\frac{r(\sigma+s)-n}{\sigma}\right)}{\sigma~\Gamma(r)}~\left(\frac{\Gamma(\beta)}{\Gamma(\beta+\alpha)}~\lambda \right)^{-\frac{n-rs}{\sigma}} ~~~~,\lambda\geq 0,\beta\in \{1\}\cup (\alpha,+\infty) \\ \ \\
			\displaystyle \frac{C^r~|\Gamma(\beta)|^r\Gamma\left(\frac{ n-rs}{\sigma}\right)\Gamma\left(\frac{r(\sigma+s)-n}{\sigma}\right)}{\sigma~\Gamma(r)}~|\lambda|^{-\frac{n-rs}{\sigma}}~,~\frac{\pi\alpha}{2}<\theta<\pi \alpha~~~~,\theta\leq |\arg(\lambda)|\leq \pi.
		\end{cases}
	\end{eqnarray*}
	Here, the constant $C>0$ -- that results from Theorem \ref{EstimatesMittagLeffler} -- is independent of $\alpha,\beta,r,s,\sigma,n$ and $\theta$.
\end{lemma}
\begin{lemma}[see Appendix \ref{AppendixB}]\label{LrMittagLeffler2}
	Let $0<\alpha\leq 1$ be given. We assume that $\sigma>0$, $r\geq 1$ and $\displaystyle s<n$ satisfy one of the following conditions:
	\begin{itemize}
		\item[{\bf (i)}] $r=1$~~$\wedge$~~ $n-\sigma<s<n$;
		\item[{\bf (ii)}] $\displaystyle \max\left\{1,\frac{n}{\sigma}\right\} <r<\infty$~~$\wedge$~~$s=0$;  
		\item[{\bf (iii)}] $\displaystyle \max\left\{1,\frac{n}{\sigma+s}\right\} <r<\frac{n}{s}$~~$\wedge$~~$0<s<n$. 
	\end{itemize}
	Then, one has
	\begin{eqnarray}
		\label{IneqD}		\left(\int_0^\infty\left|~\rho^{\sigma-s} E_{\alpha,\alpha}(-\lambda \rho^\sigma t^{\alpha})~\right|^r \rho^{n-1}d\rho\right)^{\frac{1}{r}}  \leq \nonumber \\ \frac{\left(~D_{r,\sigma,n}^{(s)}(\alpha,\lambda)~\right)^{\frac{1}{r}}}{\Gamma(\alpha)}~t^{-\frac{\alpha}{\sigma}\left(\frac{n}{r}+{\sigma-s}\right)}, 
	\end{eqnarray}
	with
	\begin{eqnarray*}
		D_{r,\sigma,n}^{(s)}(\alpha,\lambda)=\displaystyle \frac{\Gamma\left(\frac{n+r(\sigma-s)
			}{\sigma}\right)\Gamma\left(\frac{r(\sigma+s)-n}{\sigma}\right)}{\sigma~\Gamma(2r)}\left(\sqrt{\frac{\Gamma(1+\alpha)}{\Gamma(1+2\alpha)}}~\lambda\right)^{-\frac{n+r(\sigma-s)}{\sigma}} \displaystyle~~(\lambda\geq 0).
	\end{eqnarray*}
\end{lemma}

\begin{remark}
	Due to the fact that the estimates depicted in Theorem \ref{OptimalEstimatesMittagLeffler} are indeed sharp (cf.~\cite{Simon14,BS21}), one can claim that the estimates obtained in Lemma \ref{LrMittagLeffler} and Lemma \ref{LrMittagLeffler2} are sharp whenever $\lambda\geq 0$.
	The strategy considered to compute the best upper bound, in terms of the norm of weighted Lebesgue space $L^r((0,\infty),\rho^{n-1}d\rho)$, relies essentially on a straightforward application of Mellin's integral identities (\ref{MellinId}) and (\ref{MellinId2}), respectively.
\end{remark}

\begin{proposition}\label{YoungIneqProposition}
	Let $0<\alpha\leq 1$ and $\beta>0$ be given.  
	We assume that $\sigma>0$ and $1\leq r\leq 2$ satisfy one of the following conditions:
	\begin{itemize}
		\item[{\bf (a)}] $\displaystyle \sigma>n$~~$\wedge$~~$r=1$;
		\item[{\bf (b)}] $\displaystyle \sigma>\frac{n}{2}$~~$\wedge$~~$\displaystyle \max\left\{1,\frac{n}{\sigma}\right\} <r\leq 2$,
	\end{itemize}
	and $1\leq p,q,r'\leq \infty$ are such that
	\begin{center}
		$\displaystyle \frac{1}{p}+\frac{1}{r'}=\frac{1}{q}+1$ and $\displaystyle \frac{1}{r}+\frac{1}{r'}=1$.
	\end{center}
	
	Then we find the following estimates to hold: 
	\begin{eqnarray}
		\label{YoungIneqHankel}
		\displaystyle \left\|~E_{\alpha,\beta}\left(-\lambda (-\Delta)^{\frac{\sigma}{2}} t^{\alpha}\right)\varphi~\right\|_q \leq \nonumber \\
		\leq \left(~\frac{\kappa_p\kappa_{r'}}{\kappa_q}\right)^n~  \frac{\widetilde{h}_{\frac{n}{2}-1}(r,r')~ \left(~\omega_{n-1}~\right)^{\frac{1}{r'}}\left(~C_{r,\sigma,n}^{(0)}(\alpha,\beta,\lambda)~\right)^{\frac{1}{r}}}{(2\pi)^{\frac{n}{2}}\left|~\Gamma(\beta)~\right|}~t^{-\frac{\alpha}{\sigma }\cdot \frac{n}{r}}~
		\|\varphi\|_p,
		\\ \nonumber \\ \nonumber \\
		\label{YoungIneqHankel2}\displaystyle \left\|~ (-\Delta)^{\frac{\sigma}{2}}E_{\alpha,\alpha}\left(-\lambda (-\Delta)^{\frac{\sigma}{2}} t^{\alpha}\right)\varphi\right\|_q\leq \nonumber \\
		\leq \left(\frac{\kappa_p\kappa_{r'}}{\kappa_q}\right)^n~ \frac{\widetilde{h}_{\frac{n}{2}-1}(r,r')~\left(~\omega_{n-1}~\right)^{\frac{1}{r'}}\left(~D_{r,\sigma,n}^{(0)}(\alpha,\lambda)~\right)^{\frac{1}{r}}}{(2\pi)^{\frac{n}{2}}\Gamma(\alpha)}~~t^{-\alpha-\frac{\alpha}{\sigma }\cdot \frac{n}{r}}~
		\|\varphi\|_p,
	\end{eqnarray}
	whereby $\omega_{n-1},\kappa_s,\widetilde{h}_{\frac{n}{2}-1}(r,r'),C_{r,\sigma,n}^{(0)}(\alpha,\beta,\lambda)$ and $D_{r,\sigma,n}^{(0)}(\alpha,\lambda)$ stand for the constants defined in {\bf Notation \ref{SharpConstants}}.
\end{proposition}

\begin{proof}
	By applying Young's inequality (\ref{YoungIneq}) to the convolution formula (\ref{MittagLefflerConvolution}), we obtain that
	$$
	\left\|~ (-\Delta)^{\frac{\eta}{2}}E_{\alpha,\beta}\left(~-\lambda(-\Delta)^{\frac{\sigma}{2}} t^\alpha~\right)\varphi~\right\|_q \leq \left(\frac{\kappa_p\kappa_{r'}}{\kappa_q}\right)^n\left\| ~{\bm K}_{\sigma,n}^\eta(t,\cdot|\alpha,\beta,\lambda)~\right\|_{p'}~\|\varphi \|_p 
	$$
	is satisfied for every $1\leq p,q,r'\leq \infty$ such that $\displaystyle \frac{1}{p}+\frac{1}{r'}=\frac{1}{q}+1$, whereby $\displaystyle \kappa_s$ stands for the sharp constant arising from Young's inequality (see {\bf Notation \ref{SharpConstants}}).
	
	Here, we observe that the reformulation of ${\bm K}_{\sigma,n}^\eta(t,\cdot|\alpha,\beta,\lambda)$, obtained in eq.~(\ref{MittagLefflerHankel}) in terms of the Hankel transform $\widetilde{\mathcal{H}}_{\frac{n}{2}-1}$, gives rise to the norm equality
	\begin{eqnarray*}
		\left\| ~{\bm K}_{\sigma,n}^\eta(t,\cdot|\alpha,\beta,\lambda)~\right\|_{r'}=\frac{1}{(2\pi)^{\frac{n}{2}}}\left(\int_{\BR^n}\left|\widetilde{\mathcal{H}}_{\frac{n}{2}-1}\left[~|x|^\eta E_{\alpha,\beta}\left(-\lambda |x|^\sigma t^{\alpha}\right)~\right]\right|^{r'} dx\right)^{\frac{1}{r'}}.
	\end{eqnarray*}
	
	Also, from eq. (\ref{MittagLefflerHankel}) we infer that $x\mapsto \left|\widetilde{\mathcal{H}}_{\frac{n}{2}-1}\left[~|x|^\eta E_{\alpha,\beta}\left(-\lambda |x|^\sigma t^{\alpha}\right)~\right]\right|^{r'}$ is radially symmetric. Thereby, the change of variables to spherical coordinates gives rise to 
	\begin{eqnarray*}
		\left\| ~{\bm K}_{\sigma,n}^\eta(t,\cdot|\alpha,\beta,\lambda)~\right\|_{r'}=\frac{\left(\omega_{n-1}\right)^{\frac{1}{r'}}}{(2\pi)^{\frac{n}{2}}}\left(\int_{0}^\infty \left|\widetilde{\mathcal{H}}_{\frac{n}{2}-1}\left[~\rho^\eta E_{\alpha,\beta}\left(-\lambda \rho^\sigma t^{\alpha}\right)~\right]\right|^{r'} \rho^{n-1}d\rho\right)^{\frac{1}{r'}},
	\end{eqnarray*}
	where $\displaystyle \omega_{n-1}$ denotes the $(n-1)-$dimensional measure of the sphere $\mathbb{S}^{n-1}$ (see {\bf Notation \ref{SharpConstants}}). 
	Then, under the assumption that one of the conditions {\bf (a)} or {\bf (b)} on the statement of Proposition \ref{YoungIneqProposition} is fulfilled, the inequality
	\begin{eqnarray}
		\label{LrNormKernel}\left\| ~{\bm K}_{\sigma,n}^\eta(t,\cdot|\alpha,\beta,\lambda)~\right\|_{r'}&\leq& \frac{\widetilde{h}_{\frac{n}{2}-1}(r,r')\left(~\omega_{n-1}~\right)^{\frac{1}{r'}}}{(2\pi)^{\frac{n}{2}}}\times \nonumber\\ &\times& \left(\int_0^\infty\left|~\rho^{\eta}E_{\alpha,\beta}(-\lambda \rho^\sigma t^{\alpha})~\right|^r\rho^{n-1}d\rho\right)^{\frac{1}{r}}	
	\end{eqnarray}
	results from the fact that $$\widetilde{\mathcal{H}}_{\frac{n}{2}-1}:L^r((0,\infty),\rho^{n-1}d\rho)\rightarrow L^{r'}((0,\infty),\rho^{n-1}d\rho)$$ is bounded for values of $1\leq r\leq 2\leq r'\leq \infty$ satisfying $\displaystyle \frac{1}{r}+\frac{1}{r'}=1$.
	This along with the norm estimates provided by Lemma \ref{LrMittagLeffler} \& Lemma \ref{LrMittagLeffler2} gives rise to the inequalities (\ref{YoungIneqHankel}) and (\ref{YoungIneqHankel2}). Indeed, the conditions $r=1$, $\sigma>n$ or
	\begin{eqnarray*}
		\displaystyle  \max\left\{1,\frac{n}{\sigma}\right\}<r\leq 2 ~~\wedge~~~\displaystyle \sigma >\dfrac{n}{2} ~&\left(~\displaystyle ~\max\left\{1,\frac{n}{\sigma}\right\}\geq \frac{n}{\sigma}~\right),
	\end{eqnarray*} fixed by assumptions {\bf (a)} and {\bf (b)} of Proposition \ref{YoungIneqProposition}, respectively, ensures that the inequality $\displaystyle r-\frac{n}{\sigma}>0$ is always fulfilled so that the boundedness of the right hand side of (\ref{LrNormKernel}) follows from the estimates obtained in Lemma \ref{LrMittagLeffler} (case of $s=0$ and $\eta=0$) and Lemma \ref{LrMittagLeffler2} (case of $s=0$ and $\eta=\sigma$).
\end{proof}

\begin{remark}
	The proof of Proposition \ref{YoungIneqProposition} encompasses part of the $L^p-L^q$ decay estimates obtained by Pham et al in \cite[Section 2.]{PMR15} (case of $\alpha=\beta=1$) and Kemppainen et al in \cite[Section 3.1]{KSZ17} (case of $\sigma=2$). We note here that our case roughly works with the same arguments as in Pham et al. with exception of a slightly technical condition: it is not necessary to assume, {\it a priori}, parity arguments encoded by the dimension of the Euclidean space $\BR^n$ (see \cite[pp.~573-576]{PMR15} for further comparisons).
\end{remark}

\begin{remark}
	Although our technique of proof, depicted throughout Subsection \ref{LpLqEstimates}, is complementary to Kemppainen et al technique (cf.~\cite[Section 5]{KSZ17}), it allows us to rid one of the majors stumbling blocks on the computation of $L^{r'}-$decay estimates for the fundamental solution (the kernel function (\ref{MittagLefflerKernel}) in our case), whose technicality of proofs heavily relies on the asymptotics of Fox-H functions (see \cite[Section 3]{KSZ17}).
\end{remark}

\subsection{$\dot{W}^{\gamma,p}-L^q$ and $\dot{W}^{\gamma,p}-\dot{W}^{\sigma,q}$ decay estimates}\label{SobolevEstimates}

Now, we turn our attention to the {\it homogeneous Sobolev spaces} discussed in the end of Subsection \ref{FourierAnalysisSetup}. Although the next proposition follows {\it mutatis mutandis} the technique of proof considered in Proposition \ref{YoungIneqProposition}, we provide below a self-contained proof of it for reader's~convenience.

\begin{proposition}\label{SobolevYoungProposition}
	Let $0<\alpha\leq 1$ and $\beta>0$ be given.
	We assume that $\sigma>0$, $0<\gamma<n$ and $1\leq r\leq 2$ satisfy one of the following conditions:
	\begin{itemize}
		\item[{\bf (i)}] $0<\gamma< n$~~$\wedge$~~$\sigma>n-\gamma$~~$\wedge$~~$r=1$;
		\item[{\bf (ii)}] $\displaystyle 0<\gamma< \frac{n}{2}$~~$\wedge$~~$\displaystyle \sigma>\frac{n}{2}-\gamma$~~$\wedge$~~$\displaystyle \max\left\{1,\frac{n}{\sigma+\gamma}\right\} <r\leq 2$; 
		\item[{\bf (iii)}] $\displaystyle \frac{n}{2}\leq \gamma< n$~~$\wedge$~~$\sigma>0$~~$\wedge$~~$\displaystyle \max\left\{1,\frac{n}{\sigma+\gamma}\right\} <r<\frac{n}{\gamma}$, 
	\end{itemize}
	and $1\leq p,q,r'\leq \infty$ are such that
	\begin{center}
		$\displaystyle \frac{1}{p}+\frac{1}{r'}=\frac{1}{q}+1$ and $\displaystyle \frac{1}{r}+\frac{1}{r'}=1$.
	\end{center}
	Then we find the following estimates to hold: 
	
	\begin{eqnarray}
		\label{NoSobolevYoungIneqHankel}
		\displaystyle \left\|~ E_{\alpha,\beta}\left(~-\lambda (-\Delta)^{\frac{\sigma}{2}} t^{\alpha}~\right)\varphi~\right\|_q\nonumber \\ \leq \left(\frac{\kappa_p\kappa_{r'}}{\kappa_q}\right)^n~\frac{\widetilde{h}_{\frac{n}{2}-1}(r,r')\left(~\omega_{n-1}~\right)^{\frac{1}{r'}}\left(~C_{r,\sigma,n}^{(\gamma)}(\alpha,\beta,\lambda)~\right)^{\frac{1}{r}}}{(2\pi)^{\frac{n}{2}}\left|~\Gamma(\beta)~\right|}~\times \nonumber  \\ \times ~t^{-\frac{\alpha}{\sigma}\left(\frac{n}{r}-\gamma\right)}~
		\left\|~(-\Delta)^{\frac{\gamma}{2}}\varphi~\right\|_p, \\ \nonumber \\ \nonumber \\
		\displaystyle \left\|~ (-\Delta)^{\frac{\sigma}{2}}E_{\alpha,\alpha}\left(~-\lambda (-\Delta)^{\frac{\sigma}{2}} t^{\alpha}~\right)\varphi~\right\|_q\leq \nonumber \\
		\label{NoSobolevYoungIneqHankel2} \leq \left(\frac{\kappa_p\kappa_{r'}}{\kappa_q}\right)^n~\frac{\widetilde{h}_{\frac{n}{2}-1}(r,r')\left(~\omega_{n-1}~\right)^{\frac{1}{r'}}\left(~D_{r,\sigma,n}^{(\gamma)}(\alpha,\beta,\lambda)~\right)^{\frac{1}{r}}}{(2\pi)^{\frac{n}{2}}\Gamma(\alpha)}~\times \nonumber \\ \times~ t^{-\alpha-\frac{\alpha}{\sigma}\left(\frac{n}{r}-\gamma\right)}
		\left\|~(-\Delta)^{\frac{\gamma}{2}}\varphi~\right\|_p.
	\end{eqnarray}
	\vspace{0.3cm}
	
	Hereby $\omega_{n-1},\kappa_s,\widetilde{h}_{\frac{n}{2}-1}(r,r'),C_{r,\sigma,n}^{(\gamma)}(\alpha,\beta,\lambda)$ and $D_{r,\sigma,n}^{(\gamma)}(\alpha,\lambda)$ stand for the constants defined in {\bf Notation \ref{SharpConstants}}.
\end{proposition}

\begin{proof}
	
	For the proof of the estimates (\ref{NoSobolevYoungIneqHankel}) and (\ref{NoSobolevYoungIneqHankel2}), we recall that the convolution representation of (\ref{MittagLefflerOperator}), provided by eq. (\ref{MittagLefflerConvolution}), can be reformulated as follows:
	\begin{eqnarray*}
		\label{MittagLefflerConvolutionSobolev}
		(-\Delta)^{\frac{\eta}{2}}E_{\alpha,\beta}\left(~-\lambda(-\Delta)^{\frac{\sigma}{2}} t^\alpha~\right)\varphi(x)=\displaystyle \left({\bm K}_{\sigma,n}^{\eta-\gamma}(t,\cdot|\alpha,\beta,\lambda)\ast (-\Delta)^{\frac{\gamma}{2}}\varphi\right)(x).
	\end{eqnarray*}
	
	Then, by a straightforward application Young's inequality (\ref{YoungIneq}) there holds
	\begin{eqnarray}
		\label{YoungIneqGamma}
		\left\|~ (-\Delta)^{\frac{\eta}{2}}E_{\alpha,\beta}\left(~-\lambda(-\Delta)^{\frac{\sigma}{2}} t^\alpha~\right)\varphi~\right\|_q \leq \nonumber \\ \leq \left(\frac{\kappa_p\kappa_{r'}}{\kappa_q}\right)^n \left\| ~{\bm K}_{\sigma,n}^{\eta-\gamma}(t,\cdot|\alpha,\beta,\lambda)~\right\|_{r'}~\left\|~(-\Delta)^{\frac{\gamma}{2}}\varphi~\right\|_p.
	\end{eqnarray}
	
	On the other hand, to prove the boundedness of $\left\| ~{\bm K}_{\sigma,n}^{\eta-\gamma}(t,\cdot|\alpha,\beta,\lambda)~\right\|_{r'}$
	we observe that in the view of the boundedness of the Hankel transform: $$\widetilde{\mathcal{H}}_{\frac{n}{2}-1}:L^r((0,\infty),\rho^{n-1}d\rho)\rightarrow L^{r'}((0,\infty),\rho^{n-1}d\rho),$$  for values of $1\leq r\leq 2\leq r'\leq \infty$ satisfying $\displaystyle \frac{1}{r}+\frac{1}{r'}=1$ (cf.~\cite[Proposition 2.1.]{DeCarli08}), and of Lemma \ref{LrMittagLeffler} \& Lemma \ref{LrMittagLeffler2}, the inequality 
	\begin{eqnarray*}
		\left\| ~{\bm K}_{\sigma,n}^{\eta-\gamma}(t,\cdot|\alpha,\beta,\lambda)~\right\|_{r'}&\leq& \frac{\widetilde{h}_{\frac{n}{2}-1}(r,r')\left(~\omega_{n-1}~\right)^{\frac{1}{r'}}}{(2\pi)^{\frac{n}{2}}}\times\\
		&\times&\left(\int_0^\infty\left|~\rho^{\eta-\gamma}E_{\alpha,\beta}(-\lambda \rho^\sigma t^{\alpha})~\right|^{r}\rho^{n-1}d\rho\right)^{\frac{1}{r}},
	\end{eqnarray*}
	is achieved for $s=\gamma$, with $0<\gamma<n$, whenever one of the conditions ${\bf (i)},{\bf (ii)}$ and ${\bf (iii)}$ assumed on the statement of {\bf Proposition \ref{SobolevYoungProposition}}, hold.
	
	Thus, the estimates (\ref{NoSobolevYoungIneqHankel}) and (\ref{NoSobolevYoungIneqHankel2}) are then immediate from the combination of the {\it sharp inequality} (\ref{YoungIneqGamma}) with the estimates provided by Lemma \ref{LrMittagLeffler} and Lemma \ref{LrMittagLeffler2}, respectively.
\end{proof}

\begin{remark}
	Despite the estimates obtained in Lemma \ref{LrMittagLeffler} and Lemma \ref{LrMittagLeffler2} (which are sharp in case of $\lambda\geq 0$), we also have adopted the sharp estimates obtained by Lieb and Loss in \cite[Theorem 4.2]{LiebLoss01} (Young's inequality) and the $L^r-L^{r'}$ estimates for the Hankel transform to achieve Proposition \ref{YoungIneqProposition} and, subsequently, Proposition \ref{SobolevYoungProposition}. Moreover, one can infer that such estimates are optimal for values of $1<r\leq 2$, accordingly to \cite[Proposition 2.1.]{DeCarli08}.
\end{remark}

\begin{remark}
	Worthy of mention, from the estimates obtained in Lemma \ref{LrMittagLeffler} and Lemma \ref{LrMittagLeffler2} it was not necessary to investigate the Mellin integral representation of the kernel (\ref{MittagLefflerKernel}) as a Fox-H function (see, for instance, the proof of \cite[Lemma 3.3]{KSZ17} for further comparisons). In our case, the mapping properties of the Hankel transform allowed us to downsize the computation of the [optimal] endpoint estimates from the decay of the generalized Mittag-Leffler functions $E_{\alpha,\beta}$ (cf. Theorem \ref{OptimalEstimatesMittagLeffler}). Such approach goes towards the optimal bound argument considered by Kemppainen et al in \cite[Remark 3.1]{KSVZ16}.
\end{remark}

\subsection{An amalgamation of dispersive estimates}\label{Overview}

Let us now amalgamate the main conclusions provided by Proposition \ref{YoungIneqProposition} and Proposition \ref{SobolevYoungProposition}.
For a sake of readability of the dispersive estimates to be considered on the sequel, we adopt from now on the shorthand notations 
\begin{eqnarray*}
	\|v(t,\cdot)\|_q\lesssim t^{\eta}~\|\varphi\|_p & \mbox{resp.}& \|v(t,\cdot)\|_q\lesssim t^{\eta}~\left\|~(-\Delta)^{\frac{\gamma}{2}}\varphi~\right\|_p 
\end{eqnarray*} 
to express the set of inequalities
\begin{eqnarray*}
	\|v(t,\cdot)\|_q\leq At^{\eta}~\|\varphi\|_p & \mbox{resp.} & \|v(t,\cdot)\|_q\leq B t^{\eta}~\left\|~(-\Delta)^{\frac{\gamma}{2}}\varphi~\right\|_p, 
\end{eqnarray*}
for some $A>0$, $B>0$ and $\eta \in \mathbb{R}$. The following regions of the
$Opq$ plane will also be adopted to abridge the exposition of the results.\newpage 
\begin{definition}
	For fixed values of $\varepsilon>0$ and $\nu>0$, we define the regions $\mathcal{R}_{1,2}(\varepsilon)$ and $\mathcal{R}_{3}(\nu,\varepsilon)$ of the $Opq$ plane in the following way:
	\begin{itemize}
		\item[~]{\bf Region $\mathcal{R}_{1,2}(\varepsilon)$}:\newline $(p,q)\in \mathcal{R}_{1,2}(\varepsilon)$ if, and only if, $\displaystyle \left(\frac{n}{p},\frac{n}{q}\right)$ satisfies the set of conditions
		\begin{eqnarray}
			\label{Ogamma12}
			\displaystyle \frac{n}{2}\leq  \frac{n}{p}<\min\left\{n,\varepsilon\right\} ~\wedge~ \displaystyle \max\left\{0,\frac{n}{p}-\min\left\{n,\varepsilon\right\}\right\}< \frac{n}{q}\leq\frac{n}{p}-\frac{n}{2}.		\end{eqnarray}
		\item[~]{\bf Region $\mathcal{R}_{3}(\nu,\varepsilon)$}:\newline $(p,q)\in \mathcal{R}_3(\nu,\varepsilon)$ if, and only if, $\displaystyle \left(\frac{n}{p},\frac{n}{q}\right)$ satisfies the set of conditions
		\begin{eqnarray}
			\label{Ogamma3}	\displaystyle \nu< \frac{n}{p}<\min\left\{n,\varepsilon\right\} ~\wedge~ \displaystyle \max\left\{0,\frac{n}{p}-\min\left\{n,\varepsilon\right\}\right\}< \frac{n}{q}\leq\frac{n}{p}-\nu.	\end{eqnarray}
	\end{itemize}
\end{definition}

The following assumption will also be required on the sequel:
\begin{assumption}\label{AssumptionLambda}
	We assume that the constant $\lambda\in \mathbb{C}$ satisfy one of the following conditions:
	\begin{itemize}
		\item[\bf (a)] $\lambda>0$;
		\item[\bf (b)] $\theta \leq |\arg(\lambda)|\leq \pi$ and $\frac{\pi\alpha}{2}<\theta<\pi \alpha$.
	\end{itemize}
\end{assumption}

\begin{corollary}[see Appendix {\ref{AppendixC}}]\label{DispersiveCorollary} Let $0<\alpha\leq 1$, $\beta>0$ be given. Assume further that {\bf Assumption \ref{AssumptionLambda}} holds for the constant $\lambda\in \mathbb{C}$.
	Then, for every $\varphi\in \mathcal{S}(\BR^n)$ there holds  \begin{eqnarray*}
		E_{\alpha,\beta}\left(~-\lambda (-\Delta)^{\frac{\sigma}{2}} t^{\alpha}~\right)\varphi \in L^q(\BR^n)&\wedge&E_{\alpha,\alpha}\left(~-\lambda (-\Delta)^{\frac{\sigma}{2}} t^{\alpha}~\right)\varphi \in \dot{W}^{\sigma,q}(\BR^n)
	\end{eqnarray*} 
	whenever the following assertions are then true:
	\begin{itemize}
		\item[\bf (1)] {\bf -- Proposition \ref{YoungIneqProposition}:} $\varphi\in L^p(\BR^n)$ and 	\begin{eqnarray}
			\label{CaseIEstimates}
			\left\|~ (-\Delta)^{\frac{\eta}{2}}E_{\alpha,\beta}\left(~-\lambda(-\Delta)^{\frac{\sigma}{2}} t^\alpha~\right)\varphi~\right\|_q \lesssim \nonumber \\  \lesssim \begin{cases}
				\displaystyle t^{-\frac{\alpha }{\sigma}\left(\frac{n}{p}-\frac{n}{q}\right)}~
				\|\varphi\|_p &, \eta=0\\  \ \\
				\displaystyle t^{-\alpha-\frac{\alpha }{\sigma}\left(\frac{n}{p}-\frac{n}{q}\right)}~
				\|\varphi\|_p &, \beta=\alpha~;~\eta=\sigma,
			\end{cases} 
		\end{eqnarray} 
		for tuples $(\sigma,p,q)$ satisfying the following conditions:
		\begin{itemize}
			\item[{\bf (a)}] {\bf $L^1-L^\infty$ \& $L^1-\dot{W}^{\sigma,\infty}$ estimates (first case):} 
			\begin{center}
				$\sigma>n$~~$\wedge$~~$(p,q)=(1,\infty)$;
			\end{center}
			\item[{\bf (b)}] {\bf $L^p-L^q$ \& $L^p-\dot{W}^{\sigma,q}$ estimates (second case):} 
			\begin{center}
				$\displaystyle \sigma>\frac{n}{2}$~~$\wedge$~~$(p,q)\in \mathcal{R}_{1,2}(\sigma)$~[~see eq. (\ref{Ogamma12})~].
			\end{center}
		\end{itemize}
		\item[\bf (2)] {\bf -- Proposition \ref{SobolevYoungProposition}:} $\varphi\in \dot{W}^{\gamma,p}(\BR^n)$  and \begin{eqnarray}
			\label{CaseIVIIIEstimates}
			\left\|~ (-\Delta)^{\frac{\eta}{2}}E_{\alpha,\beta}\left(~-\lambda(-\Delta)^{\frac{\sigma}{2}} t^\alpha~\right)\varphi~\right\|_q \lesssim \nonumber \\  \lesssim \begin{cases}
				\displaystyle t^{-\frac{\alpha}{\sigma}\left(\frac{n}{p}-\frac{n}{q}-\gamma\right)}~
				\left\|~(-\Delta)^{\frac{\gamma}{2}}\varphi~\right\|_p &, \eta=0\\  \ \\
				\displaystyle t^{-\alpha-\frac{\alpha}{\sigma}\left(\frac{n}{p}-\frac{n}{q}-\gamma\right)}
				\left\|~(-\Delta)^{\frac{\gamma}{2}}\varphi~\right\|_p &, \beta=\alpha~;~\eta=\sigma,
			\end{cases}
		\end{eqnarray} 	
		for tuples $(\sigma,\gamma,p,q)$ satisfying the following conditions:
		\begin{itemize}
			\item[{\bf (i)}] {\bf $\dot{W}^{\gamma,1}-L^\infty$ \& $\dot{W}^{\gamma,1}-\dot{W}^{\sigma,\infty}$ estimates (first case):} 
			\begin{center}
				$0<\gamma< n$~~$\wedge$~~$\sigma>n-\gamma$~~$\wedge$~~$(p,q)=(1,\infty)$.
			\end{center}
			\item[{\bf (ii)}] {\bf $\dot{W}^{\gamma,p}-L^p$ \& $\dot{W}^{\gamma,p}-\dot{W}^{\sigma,q}$ estimates (second case):} 
			\begin{center}
				$\displaystyle 0<\gamma< \frac{n}{2}$~~$\wedge$~~$\displaystyle \sigma>\frac{n}{2}-\gamma$~~$\wedge$~~$(p,q)\in \mathcal{R}_{1,2}(\sigma+\gamma)$\newline  [~see eq. (\ref{Ogamma12})~].
			\end{center}
			\item[{\bf (iii)}]{\bf $\dot{W}^{\gamma,p}-L^p$ \& $\dot{W}^{\gamma,p}-\dot{W}^{\sigma,q}$ estimates (third case):} 
			\begin{center}
				$\sigma>0$~~$\wedge$~~$\displaystyle \frac{n}{2}\leq \gamma\leq n$~~$\wedge$~~$(p,q)\in \mathcal{R}_{3}(\gamma,\sigma+\gamma)$ \newline [~see eq. (\ref{Ogamma3})~]. 
			\end{center}
		\end{itemize}
	\end{itemize}
\end{corollary}

\newpage 

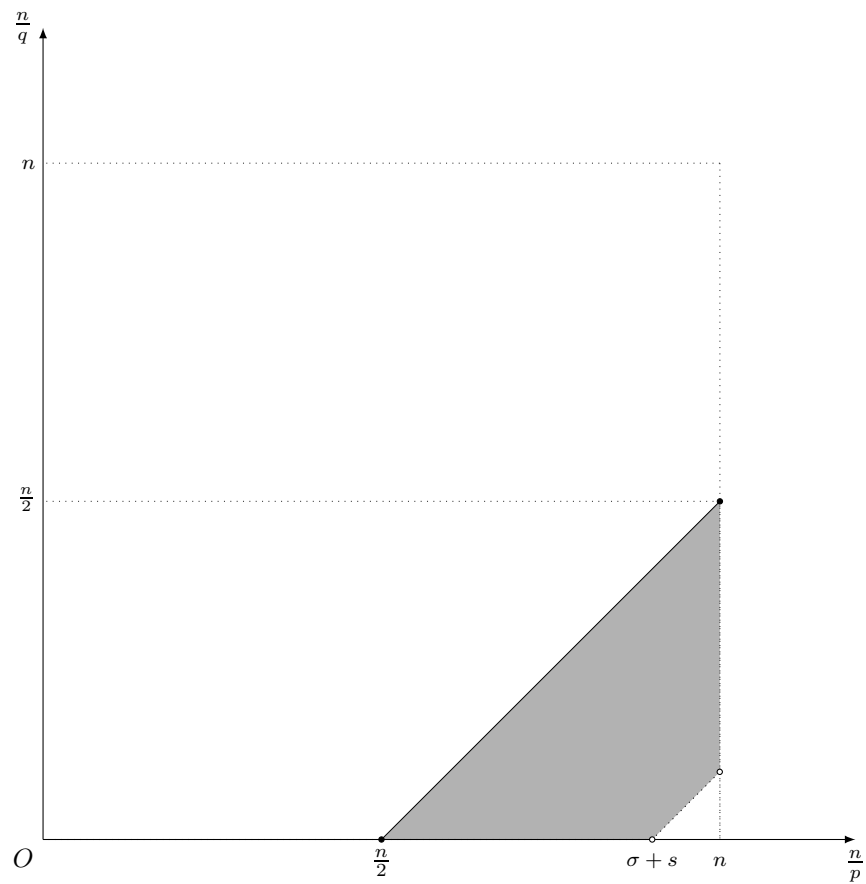
\begin{figure}\label{EstimatePicture}
	\begin{tikzpicture}[scale=9]
		\tkzInit[xmin=0, xmax=.7, ymin=0, ymax=.70]
		\tkzDrawX[noticks, label=\(\frac{n}{p}\)]
		\tkzDrawY[noticks, label=\( \frac{n}{q}\)]
		\tkzDefPoints{0/0/O}\tkzLabelPoints[below left](O)
		\tkzDefPoints{	0.5/0/A,
			1/.5/B,
			1/1/C,
			0.5/1/D}
		
		\tkzDefPoint(0.5,0){E}
		\tkzDefPoint(1,0.5){F}
		\tkzDefPoint(0.1,0){M}
		\tkzDefPoint(0.9,0){M'}
		\tkzDefPoint(0.35,0){Q}
		\tkzDefPoint(0.35,1){R}
		\tkzDefPoint(0.35,1.1){S}
		\tkzDefPoint(1,0.9){P}
		\tkzDefPoint(1,0.1){P'}
		
		\tkzDefPoint(0.5,0.4){N}
		\tkzDefPointWith[colinear=at M'](O,C)\tkzGetPoint{aux}

		\tkzInterLL(M,N)(Q,R)\tkzGetPoint{I1}
		
		\tkzInterLL(M',aux)(B,C)\tkzGetPoint{M''}
		\tkzInterLL(M',M'')(A,D)\tkzGetPoint{I2}		
		
		\tkzPointShowCoord[xlabel=\(n\), ylabel=\(n\), xstyle={below=2mm},ystyle={left=1mm}](C)
		\tkzPointShowCoord[noxdraw,noydraw, ylabel=\(\frac{n}{2}\), xstyle={below=2mm},ystyle={left=1mm}](B)
		
		\tkzPointShowCoord[xlabel=\(\sigma+s\), xstyle={below=2mm},ystyle={left=1mm}](M')
		
		
		
		\tkzPointShowCoord[ xstyle={below=2mm},ystyle={left=1mm}](B)

		\tkzDrawSegments(A,B)
		\tkzDrawSegments[dotted](M',M'')
		
		\tkzDrawPoints[fill=white](M',M'')
		\tkzDrawPoints[fill=black](A,B)
		\tkzLabelPoint[below](A){\(\frac{n}{2}\)}
		
		
		\begin{scope}[on background layer]
			\tkzFillPolygon[gray!60](A,B,M'',M')
		\end{scope}
	\end{tikzpicture}
	\caption{Geometric interpretation of {\bf Corollary \ref{DispersiveCorollary}} under the constraints $s\in \{0,\gamma\}$ and $\frac{n}{2}<\sigma+s<n$. In this case, we are unable to obtain $\dot{W}^{s,1}-L^\infty$ and $\dot{W}^{s,1}-\dot{W}^{\sigma,\infty}$ estimates.}
\end{figure}

\begin{figure}\label{EstimatePictureb}
	\begin{tikzpicture}[scale=9]
		\tkzInit[xmin=0, xmax=.7, ymin=0, ymax=.70]
		\tkzDrawX[noticks, label=\(\frac{n}{p}\)]
		\tkzDrawY[noticks, label=\( \frac{n}{q}\)]
		\tkzDefPoints{0/0/O}\tkzLabelPoints[below left](O)
		\tkzDefPoints{	0.5/0/A,
			1/.5/B,
			1/1/C,
			0.5/1/D}
		
		\tkzDefPoint(0.5,0){E}
		\tkzDefPoint(1,0.5){F}
		\tkzDefPoint(0.1,0){M}
		\tkzDefPoint(0.9,0){M'}
		\tkzDefPoint(0.35,0){Q}
		\tkzDefPoint(0.35,1){R}
		\tkzDefPoint(0.35,1.1){S}
		\tkzDefPoint(1,0.9){P}
		\tkzDefPoint(1,0){P'}
		
		\tkzDefPoint(0.5,0.4){N}
		\tkzDefPointWith[colinear=at M'](O,C)\tkzGetPoint{aux}

		\tkzInterLL(M,N)(Q,R)\tkzGetPoint{I1}
		
		\tkzInterLL(M',aux)(B,C)\tkzGetPoint{M''}
		\tkzInterLL(M',M'')(A,D)\tkzGetPoint{I2}		
		
		\tkzPointShowCoord[xlabel=\(n\), ylabel=\(n\), xstyle={below=2mm},ystyle={left=1mm}](C)
		\tkzPointShowCoord[noxdraw,noydraw, ylabel=\(\frac{n}{2}\), xstyle={below=2mm},ystyle={left=1mm}](B)
		
		
		
		
		\tkzPointShowCoord[ylabel=\(\frac{n}{2}\), xstyle={below=2mm},ystyle={left=1mm}](B)

		\tkzDrawSegments(A,B)
		\tkzDrawSegments(P',B)
		
		\tkzDrawPoints[fill=black](A,B,P')
		\tkzLabelPoint[below](A){\(\frac{n}{2}\)}
		
		
		\begin{scope}[on background layer]
			\tkzFillPolygon[gray!60](A,B,P')
		\end{scope}
	\end{tikzpicture}
	\caption{Geometric interpretation of {\bf Corollary \ref{DispersiveCorollary}} under the constraints $s\in \{0,\gamma\}$ and $\sigma+s\geq n$. In this case, we also got $\dot{W}^{s,1}-L^\infty$ and $\dot{W}^{s,1}-\dot{W}^{\sigma,\infty}$ estimates.}
\end{figure}
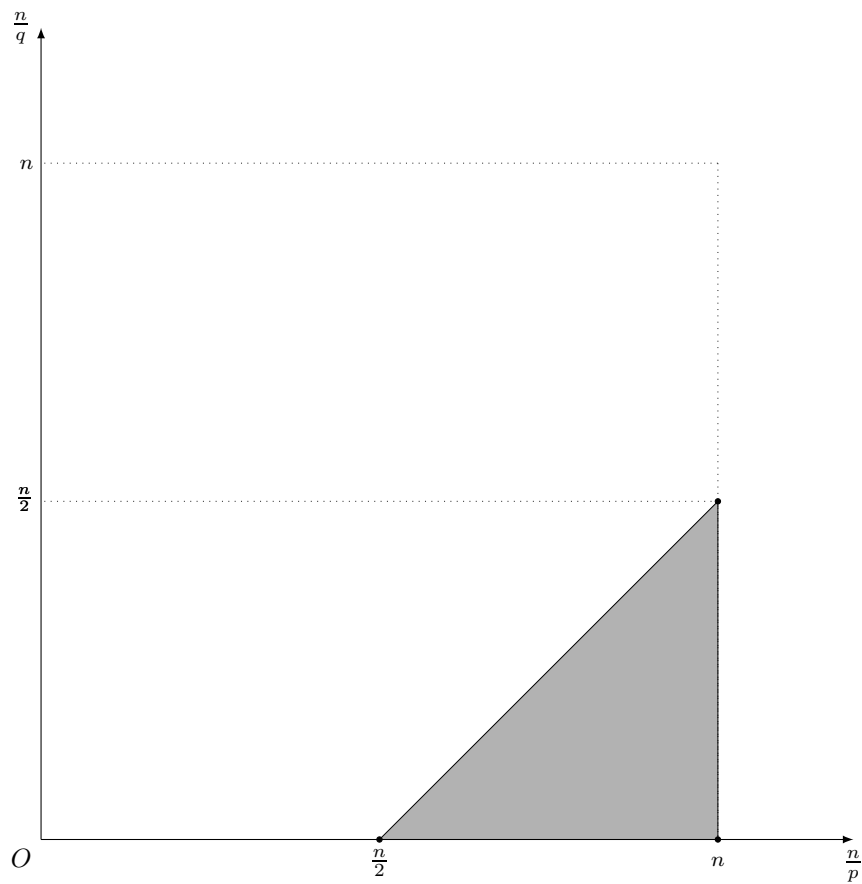

\newpage 
~
\newpage
~
\newpage 

\section{Main Results}\label{MainSection}

Let us turn at last our attention to the {\bf Cauchy Problem \ref{CP_parabolic}} and {\bf Cauchy Problem \ref{CP_hyperbolic}}. 

\subsection{Notations}

To state a solution representation for both problems, the auxiliar constants
\begin{eqnarray}
	\label{RootsLambda}\lambda_{\pm} =\begin{cases}
		\displaystyle \frac{\mu}{2}\pm i\sqrt{1-\frac{\mu^2}{4}}&, 0< \mu<2 \\ \ \\
		\displaystyle \frac{\mu}{2}\pm \sqrt{\frac{\mu^2}{4}-1} &, \mu\geq 2 
	\end{cases}
\end{eqnarray}
as well as the kernel functions ${\bm M}_{\sigma,n}(t,x|\alpha,\beta,\mu)$, ${\bm N}_{\sigma,n}(t,x|\alpha,\mu)$ and ${\bm J}_{\sigma,n}(t,x|\alpha,\mu)$, introduced as below, will be required on the proof of Theorem \ref{Theorem1}, Theorem \ref{Theorem2} and elsewhere.

\begin{notation}[Kernel functions underlying to {\bf Cauchy Problems \ref{CP_parabolic} \& \ref{CP_hyperbolic}}]\label{MarquesNelsonKernels}
	\begin{eqnarray*}
		{\bm M}_{\sigma,n}(t,x|\alpha,\beta,\mu)=\\ \ \\
		\begin{cases}
			\displaystyle 	\frac{t^{\beta-1}}{(2\pi)^n} \int_{\BR^n}~\frac{\lambda_+E_{\alpha,\beta}(-\lambda_+|\xi|^\sigma t^\alpha)-\lambda_-E_{\alpha,\beta}(-\lambda_-|\xi|^\sigma t^\alpha)}{\lambda_+-\lambda_-}~e^{ix\cdot \xi}d\xi&,~\mu\neq 2 \\ \ \\
			\displaystyle \frac{t^{\beta-1}}{\alpha(2\pi)^n} \int_{\BR^n} \left(~E_{\alpha,\beta-1}(-|\xi|^\sigma t^{\alpha})-(1+\alpha-\beta)E_{\alpha,\beta}(-|\xi|^\sigma t^{\alpha})~\right)~e^{ix\cdot \xi}d\xi&,~\mu= 2.
		\end{cases} 
	\end{eqnarray*}
	
	\begin{eqnarray*}{\bm N}_{\sigma,n}(t,x|\alpha,\mu)=\\  \ \\
		=\begin{cases}
			\displaystyle \frac{1}{(2\pi)^n} \int_{\BR^n}	\left(\frac{E_{\alpha,1}(-\lambda_+|\xi|^\sigma t^\alpha)}{1-(\lambda_+)^2}+\frac{E_{\alpha,1}(-\lambda_-|\xi|^\sigma t^\alpha)}{1-(\lambda_-)^2}\right)e^{ix\cdot \xi}d\xi &,~\mu\neq 2 \\ \ \\
			\displaystyle \frac{1}{(2\pi)^n} \int_{\BR^n}\left(E_{\alpha,1}(-|\xi|^\sigma t^\alpha)+\frac{|\xi|^{\sigma}t^{\alpha}}{\alpha}E_{\alpha,\alpha}(-|\xi|^\sigma t^{\alpha})\right)e^{ix\cdot \xi}d\xi&,~\mu= 2.
		\end{cases} 
	\end{eqnarray*}
\end{notation}

\begin{notation}[Kernel function underlying to {\bf Cauchy Problem \ref{CP_hyperbolic}}]\label{JorgeKernel}
	\begin{eqnarray*}
		{\bm J}_{\sigma,n}(t,x|\alpha,\mu)=\\ \ \\
		=\begin{cases}
			\displaystyle 	\frac{t}{(2\pi)^n}\int_{\BR^n}\left(~\frac{\lambda_+E_{\alpha,2}(-\lambda_+|\xi|^\sigma t^\alpha)-\lambda_-E_{\alpha,2}(-\lambda_-|\xi|^\sigma t^\alpha)}{\lambda_+-\lambda_-}~\right) e^{i x\cdot \xi}d\xi&,~\mu\neq 2 \\ \ \\
			\displaystyle \frac{t}{\alpha(2\pi)^n}\int_{\BR^n}\left(~E_{\alpha,1}(-|\xi|^\sigma t^\alpha)-(\alpha-1)E_{\alpha,2}(-|\xi|^\sigma t^\alpha)~\right)e^{ix\cdot \xi}d\xi&,~\mu= 2.
		\end{cases} 
	\end{eqnarray*}
\end{notation}

\subsection{Solution of the Cauchy Problem \ref{CP_parabolic}}\label{SolutionCP_parabolic}

\begin{theorem}\label{Theorem1}
	Let $0<\alpha\leq \frac{1}{2}$, $2\alpha\leq \beta< 2\alpha+1$, $\sigma>0$ and $\mu>0$ 
	be given. 
	If $u_0\in \mathcal{S}(\mathbb{R}^n)$ and 
		$f(t,\cdot)\in \mathcal{S}(\mathbb{R}^n)$, for every $t\geq 0$,
		then $u: [0,\infty)\times \BR^n\longrightarrow \BR$ defined through \begin{eqnarray*}
			u(t,x)=\left({\bm N}_{\sigma,n}(t,\cdot|\alpha,\mu)*u_0\right)(x)+\int_{0}^t \left[~{\bm M}_{\sigma,n}(t-\tau,\cdot|\alpha,\beta,\mu)\ast f(\tau,\cdot)~\right](x)~d\tau
		\end{eqnarray*} solves the {\bf Cauchy Problem \ref{CP_parabolic}}, whereby ${\bm M}_{\sigma,n}(t,x|\alpha,\beta,\mu)$ and ${\bm N}_{\sigma,n}(t,x|\alpha,\mu)$ denote the kernel functions introduced in {\bf Notation \ref{MarquesNelsonKernels}}.
	\end{theorem} 
	
	\begin{proof}
	First, we recall that $u(t,x)$ solves the {\bf Cauchy Problem \ref{CP_parabolic}} if and only if $\widehat{u}(t,\xi)=(\mathcal{F}u)(t,\xi)$ solves
	\begin{equation}\label{CPu0}
		\begin{cases}
			\displaystyle \partial^{2\alpha}_t \widehat{u}(t,\xi)+\mu |\xi|^{\sigma}\partial^\alpha_t \widehat{u}(t,\xi)+|\xi|^{2\sigma} \widehat{u}(t,\xi)=I^{\beta-2\alpha}_{0^+}\widehat{f}(t,\xi)&,\mbox{in} ~(0,\infty)\times \mathbb{R}^n \\ \ \\
			\widehat{u}(0,\xi)=\widehat{u_0}(\xi)&,\mbox{in}~\mathbb{R}^n,
		\end{cases}
	\end{equation}
	with $\widehat{u_0}(\xi)=(\mathcal{F}u_0)(\xi)$ and $I^{\beta-2\alpha}_{0^+}\widehat{f}(t,\xi)=I^{\beta-2\alpha}_{0^+}(\mathcal{F}f)(t,\xi)$.
	
	In order to compute the solution representation of (\ref{CPu0}) we adopt the Laplace transform $\mathcal{L}$ defined through eq. (\ref{LT}). First, one obtains from eqs.~(\ref{LaplaceCaputo}) and (\ref{LaplaceRLiouville}), and from the initial conditions underlying to the Cauchy Problem (\ref{CPu0}) that
	\begin{eqnarray*}
		\mathcal{L}[\partial^{2\alpha}_t \widehat{u}(t,\xi)](s)&=&s^{2\alpha}\mathcal{L}[ \widehat{u}(t,\xi)](s)-s^{2\alpha-1}\widehat{u_0}(\xi), \\
		\mathcal{L}[\partial^{\alpha}_t \widehat{u}(t,\xi)](s)&=&s^{\alpha}\mathcal{L}[ \widehat{u}(t,\xi)](s)-s^{\alpha-1}\widehat{u_0}(\xi), \\
		\mathcal{L}[I^{\beta-2\alpha}_{0^+} \widehat{f}(t,\xi)](s)&=&s^{2\alpha-\beta}\mathcal{L}[ \widehat{f}(t,\xi)](s),
	\end{eqnarray*}
	since $0<2\alpha\leq 1$ and $0\leq \beta-2\alpha< 1$.
	
	Thus, after few algebraic manipulations we conclude that the solution of (\ref{CPu0}), $\widehat{u}(t,\xi)$, admits the following Laplace representation
	\begin{eqnarray}
		\label{Us}\mathcal{L}[\widehat{u}(t,\xi)](s)=U_0(s,\xi)~\widehat{u_0}(\xi)+M(s,\xi)~\mathcal{L}[\widehat{f}(t,\xi)](s),
	\end{eqnarray}
	with 
	\begin{eqnarray}
		\label{M}M(s,\xi)&=&\frac{s^{2\alpha-\beta}}{s^{2\alpha}+\mu |\xi|^\sigma s^{\alpha} +|\xi|^{2\sigma}}, \\ \nonumber \\
		\label{U0}U_0(s,\xi)&=&\frac{s^{2\alpha-1}+\mu |\xi|^\sigma s^{\alpha-1}}{s^{2\alpha}+\mu |\xi|^\sigma s^{\alpha} +|\xi|^{2\sigma}}. 
	\end{eqnarray}
	
	By taking into account the constants $\lambda_+$ and $\lambda_-$ defined viz eq. (\ref{RootsLambda}), we observe that $s=-\lambda_-|\xi|^\sigma$ and $s=-\lambda_+|\xi|^\sigma$ correspond to the algebraic roots of the polynomial $P(s)=s^2+\mu|\xi|^\sigma s+|\xi|^{2\sigma}$. Then
	\begin{eqnarray*}
		\label{FactorizationIds}s^{2\alpha}+\mu|\xi|^\sigma s^\alpha+|\xi|^{2\sigma}=(s^\alpha+\lambda_+|\xi|^\sigma)(s^\alpha+\lambda_-|\xi|^\sigma)
	\end{eqnarray*}
	so that for $\lambda_+\neq \lambda_-$ (i.e. for $\mu\neq 2$) the set of identities
	\begin{eqnarray}
		\label{UndeterminedCoefficients0}
		~\frac{s^\alpha}{s^{2\alpha}+\mu|\xi|^\sigma s^\alpha+|\xi|^{2\sigma}}=\frac{1}{\lambda_+-\lambda_-}\left(\frac{\lambda_+}{s^\alpha+\lambda_+|\xi|^\sigma}-\frac{\lambda_-}{s^\alpha+\lambda_-|\xi|^\sigma}\right), \\
		\label{UndeterminedCoefficients1}
		~\frac{1}{s^{2\alpha}+\mu|\xi|^\sigma s^\alpha+|\xi|^{2\sigma}}=\frac{-|\xi|^{-\sigma}}{\lambda_+-\lambda_-}\left(\frac{1}{s^\alpha+\lambda_+|\xi|^\sigma}-\frac{1}{s^\alpha+\lambda_-|\xi|^\sigma}\right)
	\end{eqnarray}
	yield straightforwardly by the method of undetermined coefficients.
	
	Thus, from eq. (\ref{UndeterminedCoefficients0}) we obtain that eq.~(\ref{M}) is equivalent to
	\begin{eqnarray}\label{MDivision}
		M(s,\xi)&=&\begin{cases}
			\displaystyle \frac{1}{\lambda_+-\lambda_-}\left(\frac{\lambda_+~s^{\alpha-\beta}}{s^\alpha+\lambda_+|\xi|^\sigma}-\frac{\lambda_-~s^{\alpha-\beta}}{s^\alpha+\lambda_-|\xi|^\sigma}\right)&,~ ~\mu\neq 2 \\ \ \\
			\displaystyle\frac{s^{2\alpha-\beta}}{(s^\alpha+|\xi|^\sigma)^2}&,~\mu= 2,
		\end{cases}
	\end{eqnarray}
	whereas the equivalence between eq.~(\ref{U0}) and 
	\begin{eqnarray}\label{U0Division}
			U_0(s,\xi)=\displaystyle \frac{1}{s}-\frac{|\xi|^{2\sigma}s^{-1}}{s^{2\alpha}+\mu |\xi|^\sigma s^{\alpha} +|\xi|^{2\sigma}} \nonumber \\ \nonumber \\
			=\begin{cases}
				\displaystyle 	\frac{1}{s}-\frac{-|\xi|^\sigma}{\lambda_+-\lambda_-}\left(\frac{s^{-1}}{s^\alpha+\lambda_+|\xi|^\sigma}-\frac{s^{-1}}{s^\alpha+\lambda_-|\xi|^\sigma}\right) &,~ \mbox{if}~\mu\neq 2 \\ \ \\
				\displaystyle \frac{1}{s}-|\xi|^{2\sigma}\frac{s^{-1}}{(s^\alpha+|\xi|^\sigma)^2}&,~\mbox{if}~\mu= 2.
			\end{cases}
	\end{eqnarray}
	follows from eq.~(\ref{UndeterminedCoefficients1}).
	
	Next, from the Laplace identity (\ref{LaplaceIdentityMittagLeffler}) and from Lemma \ref{Properties_ML} we obtain the set of identities
	\begin{eqnarray}\label{MLaplaceInverse}
			\mathcal{L}^{-1}[M(s,\xi)](t)=
			\begin{cases}
				\displaystyle 	t^{\beta-1}~\frac{\lambda_+E_{\alpha,\beta}(-\lambda_+|\xi|^\sigma t^\alpha)-\lambda_-E_{\alpha,\beta}(-\lambda_-|\xi|^\sigma t^\alpha)}{\lambda_+-\lambda_-}&,~\mu\neq 2 \\ \ \\
				\displaystyle t^{\beta-1}E_{\alpha,\beta}^2(-|\xi|^\sigma t^{\alpha})&,~\mu= 2
			\end{cases}\nonumber \\ \nonumber \\ \nonumber \\ 
			=
			\begin{cases}
				\displaystyle 	t^{\beta-1}~\frac{\lambda_+E_{\alpha,\beta}(-\lambda_+|\xi|^\sigma t^\alpha)-\lambda_-E_{\alpha,\beta}(-\lambda_-|\xi|^\sigma t^\alpha)}{\lambda_+-\lambda_-}&,~\mu\neq 2 \\ \ \\
				\displaystyle \frac{t^{\beta-1}}{\alpha}\left(~E_{\alpha,\beta-1}(-|\xi|^\sigma t^{\alpha})-(1+\alpha-\beta)E_{\alpha,\beta}(-|\xi|^\sigma t^{\alpha})~\right)&,~\mu= 2.
			\end{cases} 
	\end{eqnarray}
	
	In the same order of ideas, from (\ref{LaplaceIdentityMittagLeffler}) 
	we conclude that the Laplace inverse of (\ref{U0Division}), $\mathcal{L}^{-1}[U_0(s,\xi)](t)$, equals to
	\begin{eqnarray*}
			\begin{cases}
				\displaystyle 	1-\frac{-|\xi|^\sigma t^\alpha}{\lambda_+-\lambda_-}\left(E_{\alpha,\alpha+1}(-\lambda_+|\xi|^\sigma t^\alpha)-E_{\alpha,\alpha+1}(-\lambda_-|\xi|^\sigma t^\alpha)\right) &,~\mu\neq 2 \\ \ \\
				\displaystyle 1-|\xi|^{2\sigma}t^{2\alpha}E_{\alpha,2\alpha+1}^2(-|\xi|^\sigma t^{\alpha})&,~\mu= 2.
			\end{cases}
	\end{eqnarray*}
	
	Note already that the constants $\lambda_+$ and $\lambda_-$ defined by eq. (\ref{RootsLambda}) satisfy the set of relations
	\begin{eqnarray}
		\label{RootsProperties}\lambda_+\lambda_-=1&\wedge &\displaystyle \frac{1}{\lambda_+-\lambda_-}\left(\frac{1}{\lambda_-}-\frac{1}{\lambda_+}\right)=1.
	\end{eqnarray}
	
	Then, by a straightforward application of Lemma \ref{Properties_ML}, one has that $\mathcal{L}^{-1}[U_0(s,\xi)](t)$, computed as above, simplifies to
	\begin{eqnarray}\label{U0LaplaceInverse}
			\begin{cases}
				\displaystyle 	1-\frac{1}{\lambda_+-\lambda_-}\left(\frac{E_{\alpha,1}(-\lambda_+|\xi|^\sigma t^\alpha)-1}{\lambda_+}-\frac{E_{\alpha,1}(-\lambda_-|\xi|^\sigma t^\alpha)-1}{\lambda_-}\right) &,~\mu\neq 2 \\ \ \\
				\displaystyle 1+|\xi|^{\sigma}t^{\alpha}\left(E_{\alpha,\alpha+1}^2(-|\xi|^\sigma t^{\alpha})-E_{\alpha,\alpha+1}(-|\xi|^\sigma t^{\alpha})\right)&,~\mu= 2
			\end{cases} \nonumber \\ \nonumber \\
			=\begin{cases}
				\displaystyle 	\frac{E_{\alpha,1}(-\lambda_+|\xi|^\sigma t^\alpha)}{1-(\lambda_+)^2}+\frac{E_{\alpha,1}(-\lambda_-|\xi|^\sigma t^\alpha)}{1-(\lambda_-)^2} &,~\mu\neq 2 \\ \ \\
				\displaystyle E_{\alpha,1}(-|\xi|^\sigma t^\alpha)+\frac{|\xi|^{\sigma}t^{\alpha}}{\alpha}E_{\alpha,\alpha}(-|\xi|^\sigma t^{\alpha})&,~\mu= 2.
			\end{cases}
	\end{eqnarray}
	
	Thereafter, by applying the Laplace convolution identity (\ref{LaplaceConvolution}) one has that eq. (\ref{Us}) is equivalent to
	\begin{eqnarray}
		\label{LaplaceFsx}
		\widehat{u}(t,\xi)=\mathcal{L}^{-1}[U_0(s,\xi)](t)~\widehat{u_0}(\xi)+\int_{0}^t \mathcal{L}^{-1}[M(s,\xi)](t-\tau)~\widehat{f}(\tau,\xi)d\tau.
	\end{eqnarray}
	
	The above representation formula provides a solution for the Cauchy Problem (\ref{CPu0}), whereby the terms $\mathcal{L}^{-1}[M(s,\xi)](t)$ and $\mathcal{L}^{-1}[U_0(s,\xi)](t)$ are given by (\ref{MLaplaceInverse}) and (\ref{U0LaplaceInverse}), respectively.

	Finally, from the {\it Fourier convolution formula} (\ref{FourierConvolutionRn}) one can moreover say that
	\begin{eqnarray*}
		u(t,x)=\left({\bm N}_{\sigma,n}(t,\cdot|\alpha,\mu)*u_0\right)(x)+\int_{0}^t \left[~{\bm M}_{\sigma,n}(t-\tau,\cdot|\alpha,\beta,\mu)\ast f(\tau,\cdot)~\right](x)~d\tau,
	\end{eqnarray*}
	where ${\bm M}_{\sigma,n}(t,x|\alpha,\beta,\mu)$ and ${\bm N}_{\sigma,n}(t,x|\alpha,\mu)$ stand for the kernel functions introduced in {\bf Notation \ref{MarquesNelsonKernels}}, concluding in this way the proof of Theorem \ref{Theorem1}.
\end{proof}
	
	\subsection{Solution of the Cauchy Problem \ref{CP_hyperbolic}}\label{SolutionCP_hyperbolic}
	
	\begin{theorem}\label{Theorem2}
		Let $\frac{1}{2}<\alpha\leq 1$, $2\alpha\leq \beta< 2\alpha+1$, $\sigma>0$, $\mu>0$ 
		be given. 
		If $u_0,u_1\in \mathcal{S}(\mathbb{R}^n)$ and 
			$f(t,\cdot)\in \mathcal{S}(\mathbb{R}^n)$, for every $t\geq 0$,
			then $u: [0,\infty)\times \BR^n \longrightarrow \BR$ defined through \begin{eqnarray*}
				u(t,x)&=&({\bm N}_{\sigma,n}(t,\cdot|\alpha,\mu)*u_0)(x)+({\bm J}_{\sigma,n}(t,\cdot|\alpha,\mu)*u_1)(x) \\
				&+&\int_{0}^t \left[~{\bm M}_{\sigma,n}(t-\tau,\cdot|\alpha,\beta,\mu)\ast f(\tau,\cdot)~\right](x)~d\tau 
			\end{eqnarray*} solves the {\bf Cauchy Problem \ref{CP_hyperbolic}}, whereby ${\bm M}_{\sigma,n}(t,\cdot|\alpha,\beta,\mu)$, ${\bm N}_{\sigma,n}(t,\cdot|\alpha,\mu)$ and ${\bm J}_{\sigma,n}(t,\cdot|\alpha,\mu)$ denote the kernel functions introduced in {\bf Notation \ref{MarquesNelsonKernels}} and {\bf Notation \ref{JorgeKernel}}.
		\end{theorem} 
		
		\begin{proof}
			Following {\it mutatis mutandis} the scheme of proof adopted in Theorem \ref{Theorem1}, we recall that for $\frac{1}{2}<\alpha\leq 1$ and $2\alpha\leq \beta< 2\alpha+1$, one has that $u(t,x)$ solves the {\bf Cauchy Problem~\ref{CP_hyperbolic}} if, and only if, $\widehat{u}(t,\xi)=(\mathcal{F}u)(t,\xi)$ is a solution of 
			\begin{equation}\label{CPu0u1}
				\begin{cases}
					\displaystyle \partial^{2\alpha}_t \widehat{u}(t,\xi)+\mu |\xi|^{\sigma}\partial^\alpha_t \widehat{u}(t,\xi)+|\xi|^{2\sigma} \widehat{u}(t,\xi)=I^{\beta-2\alpha}_{0^+}\widehat{f}(t,\xi)&,\mbox{in}~(0,\infty)\times \mathbb{R}^n \\ \ \\ 
					\widehat{u}(0,\xi)=\widehat{u_0}(\xi)&,\mbox{in}~\mathbb{R}^n \\ \ \\
					\partial_t\widehat{u}(0,\xi)=\widehat{u_1}(\xi)&,\mbox{in}~\mathbb{R}^n,
				\end{cases}
			\end{equation}
			with $\widehat{u_0}(\xi)=(\mathcal{F}u_0)(\xi)$, $\widehat{u_1}(\xi)=(\mathcal{F}u_1)(\xi)$ and $I^{\beta-2\alpha}_{0^+}\widehat{f}(t,\xi)=I^{\beta-2\alpha}_{0^+}(\mathcal{F}f)(t,\xi)$. 
			
			By ensuring that $1\leq 2\alpha<2$ and $0\leq \beta-2\alpha< 1$, we note already that 
			\begin{eqnarray*}
				\mathcal{L}[\partial^{2\alpha}_t \widehat{u}(t,\xi)](s)&=&s^{2\alpha}\mathcal{L}[ \widehat{u}(t,\xi)](s)-s^{2\alpha-1}\widehat{u_0}(\xi)-s^{2\alpha-2}\widehat{u_1}(\xi), \\
				\mathcal{L}[\partial^{\alpha}_t \widehat{u}(t,\xi)](s)&=&s^{\alpha}\mathcal{L}[ \widehat{u}(t,\xi)](s)-s^{\alpha-1}\widehat{u_0}(\xi), \\
				\mathcal{L}[I^{\beta-2\alpha}_{0^+} \widehat{f}(t,\xi)](s)&=&s^{2\alpha-\beta}\mathcal{L}[ \widehat{f}(t,\xi)](s)
			\end{eqnarray*}
			results from eqs.~(\ref{LaplaceCaputo}) and (\ref{LaplaceRLiouville}), and from the initial conditions underlying to the Cauchy Problem (\ref{CPu0u1}).
			Then, $\mathcal{L}[\widehat{u}(t,\xi)](s)$ is of the form
			\begin{eqnarray}
				\label{Uss}\mathcal{L}[\widehat{u}(t,\xi)](s)=U_0(s,\xi)\widehat{u_0}(\xi)+U_1(s,\xi)\widehat{u_1}(\xi)+M(s,\xi) \widehat{f}(t,\xi),
			\end{eqnarray}
			with
			\begin{eqnarray}
				\label{U1}U_1(s,\xi)&=&\frac{s^{2\alpha-2}}{s^{2\alpha}+\mu |\xi|^\sigma s^{\alpha} +|\xi|^{2\sigma}},
			\end{eqnarray}
			and $M(s,\xi)$,~$U_0(s,\xi)$ defined viz eqs. (\ref{M}) and (\ref{U0}), respectively.
			
			Hence, the solution of the {\bf Cauchy Problem \ref{CP_hyperbolic}} is of the form
			\begin{eqnarray*}
				u(t,x)&=&({\bm N}_{\sigma,n}(t,\cdot|\alpha,\mu)*u_0)(x)+(\mathcal{F}^{-1}\mathcal{L}^{-1})[U_1(s,\xi)](t)*u_1(x) \\
				&+&\int_{0}^t \left({\bm M}_{\sigma,n}(t-\tau,\cdot|\alpha,\beta,\mu)*f(\tau,\cdot)\right)(x)~d\tau.
			\end{eqnarray*}
			
			To conclude the proof it remains us to compute, at a first glance, the function $\mathcal{L}^{-1}[U_1(s,\xi)](t)$ by means of the inverse of the Laplace transform, and afterwards, use the Fourier inversion formula, provided by eq. (\ref{FInv}), to show that $(\mathcal{F}^{-1}\mathcal{L}^{-1})[U_1(s,\cdot)](t)$ coincides with the function ${\bm J}_{\sigma,n}(t,x|\alpha,\mu)$ introduced on {\bf Notation \ref{JorgeKernel}}.

			To do so, we start to observe that $U_1(s,\xi)$ and $M(s,\xi)$ (see eq.~(\ref{M})) are interrelated by the formula $U_1(s,\xi)=s^{\beta-2}M(s,\xi)$. Then,
			\begin{eqnarray}
				\label{U1Division}
				U_1(s,\xi)=\begin{cases}
					\displaystyle \frac{1}{\lambda_+-\lambda_-}\left(\frac{\lambda_+~s^{\alpha-2}}{s^\alpha+\lambda_+|\xi|^\sigma}-\frac{\lambda_-~s^{\alpha-2}}{s^\alpha+\lambda_-|\xi|^\sigma}\right)&,~ ~\mu\neq 2 \\ \ \\
					\displaystyle\frac{s^{2\alpha-2}}{(s^\alpha+|\xi|^\sigma)^2}&,~\mu= 2
				\end{cases}
			\end{eqnarray}
			is rather immediate from eq. (\ref{MDivision}). 
			
			Thus, from the Laplace identity (\ref{LaplaceIdentityMittagLeffler}), one obtains that $\mathcal{L}^{-1}[U_1(s,\xi)](t)$ equals to
			\begin{eqnarray*}
					\begin{cases}
						\displaystyle 	t~\frac{\lambda_+E_{\alpha,2}(-\lambda_+|\xi|^\sigma t^\alpha)-\lambda_-E_{\alpha,2}(-\lambda_-|\xi|^\sigma t^\alpha)}{\lambda_+-\lambda_-}&,~\mu\neq 2 \\ \ \\
						\displaystyle tE_{\alpha,2}^2(-|\xi|^\sigma t^\alpha)&,~\mu= 2,
					\end{cases}
			\end{eqnarray*}
			
			Moreover, from direct application of Lemma \ref{Properties_ML},  $\mathcal{L}^{-1}[U_1(s,\xi)](t)$ simplifies to
			\begin{eqnarray}\label{U1LaplaceInverse}
					\begin{cases}
						\displaystyle 	t~\frac{\lambda_+E_{\alpha,2}(-\lambda_+|\xi|^\sigma t^\alpha)-\lambda_-E_{\alpha,2}(-\lambda_-|\xi|^\sigma t^\alpha)}{\lambda_+-\lambda_-}&,~\mu\neq 2 \\ \ \\
						\displaystyle \frac{t}{\alpha}\left(E_{\alpha,1}(-|\xi|^\sigma t^\alpha)-(\alpha-1)E_{\alpha,2}(-|\xi|^\sigma t^\alpha)\right)&,~\mu= 2
					\end{cases} 
			\end{eqnarray}
			Finally, by employing the Fourier inversion formula (\ref{FInv}) to the previous formula we obtain the kernel function ${\bm J}_{\sigma,n}(t,x|\alpha,\mu)$, as desired.
		\end{proof}
		
		\subsection{Strichartz estimates}\label{StrichartzSection}

		Having already gathered the main technical results in Section \ref{DispersiveSection}, we are now in a position to investigate the size and decay of the solution of {\bf Cauchy Problem \ref{CP_parabolic}} and {\bf Cauchy Problem \ref{CP_hyperbolic}}, from its closed-form representations, in terms of the celebrated Strichartz estimates (cf.~\cite[Chapter 2.]{Tao06}).
		On the formulation of the required assumptions, we need to consider the following regions of the
		$O\sigma\gamma$ plane:
		\begin{eqnarray}
			\label{S0-region}
			\mathcal{S}_0(\alpha)=\left\{(\sigma,\gamma)~:~0<\gamma< n~\wedge~n-\gamma<\sigma\leq \alpha\gamma~\right\};\\ \nonumber \\
			\label{S12-region}	\mathcal{S}_{1,2}(\alpha)=\left\{(\sigma,\gamma)~:~\displaystyle 0<\gamma< \frac{n}{2}~\wedge~\displaystyle \frac{n}{2}-\gamma<\sigma\leq \alpha\gamma~\right\};\\ \nonumber \\
			\label{S3-region}
			\mathcal{S}_{3}(\alpha)=\left\{(\sigma,\gamma)~:~\displaystyle \frac{n}{2}\leq \gamma< n~\wedge~0<\sigma\leq \alpha\gamma~\right\}.
		\end{eqnarray}		
		
		\begin{assumption}\label{AssumptionParameters}
			We assume that the $4-$tuple $(\sigma,\gamma,p,q)$ satisfy, at least, one of the following conditions:
			\begin{itemize}
				\item[\bf (i)] $(\sigma,\gamma)\in \mathcal{S}_{0}(\alpha)$~~$\wedge$~~$(p,q)=(1,\infty)$ \newline [~see eq. (\ref{S0-region})~];
				\item[\bf (ii)] $(\sigma,\gamma)\in \mathcal{S}_{1,2}(\alpha)$~~$\wedge$~~$(p,q)\in \mathcal{R}_{1,2}\left(\sigma+\varepsilon\right)\cap  \mathcal{R}_{1,2}\left(\sigma+\nu\right)$ \newline  [see eqs. (\ref{S12-region}) and (\ref{Ogamma12})];
				\item[\bf (iii)] $(\sigma,\gamma)\in \mathcal{S}_{3}(\alpha)$~~$\wedge$~~$(p,q)\in \mathcal{R}_{3}\left(\varepsilon,\sigma+\varepsilon\right)\cap \mathcal{R}_{3}\left(\nu,\sigma+\nu\right)$  \newline [~see eqs. (\ref{S3-region}) and (\ref{Ogamma3})~],
			\end{itemize}
			with
			\begin{center}
				$\displaystyle \varepsilon:=\gamma-\frac{\sigma}{\alpha}$ and $\displaystyle \nu:=\gamma-\frac{\sigma}{\alpha}\cdot \beta.$
			\end{center}
		\end{assumption}
		\begin{assumption}\label{AssumptionF}
			We assume that the following set of conditions on the function $f:[0,T]\times \BR^n\longrightarrow \BR$:
			\begin{itemize}
				\item[\bf (I)] $f(\tau,\cdot)\in \dot{W}^{\gamma-\frac{\sigma}{\alpha},p}(\BR^n)\cap \dot{W}^{\gamma-\frac{\sigma}{\alpha}\cdot\beta,p}(\BR^n)$;
				\item[\bf (II)] $f\in L^\infty_tL^q_x([0,T]\times\BR^n)$;
				\item[\bf (III)] $\displaystyle \left\|~(-\Delta)^{\frac{\gamma}{2}-\frac{\sigma}{2\alpha}\cdot\beta} f(\tau,\cdot)~\right\|_p\lesssim \tau^{\frac{\alpha}{\sigma}\left(\frac{n}{p}-\frac{n}{q}-\left(\gamma-\frac{\sigma}{\alpha}\right)\right)-1} \|f\|_{L^\infty_tL^q_x([0,T]\times\BR^n)}$,
			\end{itemize}
			hold for every $\tau\in [0,T]$.
		\end{assumption}

		We start with the following theorem, where we prove a retarded Strichartz estimate for $f(t,x)$, appearing on the solution representation of the {\bf Cauchy Problem \ref{CP_parabolic}} \& {\bf Cauchy Problem \ref{CP_hyperbolic}}. That will be crucial afterwards on the proof of Theorem \ref{Theorem1Strichartz} \& Theorem \ref{Theorem2Strichartz}, when it will be considered the initial datum/data 
		\begin{center}
			$u_0\in \dot{W}^{\gamma-\frac{\sigma}{\alpha},p}(\BR^n)$ resp. $(u_0,u_1)\in \dot{W}^{\gamma,p}(\BR^n)\times \dot{W}^{\gamma-\frac{\sigma}{\alpha},p}(\BR^n)$.
		\end{center}

		\begin{theorem}\label{fBoundedProposition}
			Let $0<\alpha\leq 1$, $2\alpha\leq\beta< 2\alpha+1$, $\mu>0$ and $\sigma>0$ be given.
			Assume further that the $4-$tuple $(\sigma,\gamma,p,q)$ satisfies {\bf Assumption \ref{AssumptionParameters}} and
			$f:[0,T]\times \BR^n\longrightarrow \BR$ satisfies
			{\bf Assumption \ref{AssumptionF}}.
			
			Then, for every $t\in[0,T]$, we find the following estimate to hold:
			\begin{eqnarray*}
				\left\|~\displaystyle \int_0^t {\bm M}_{\sigma,n}(t-\tau,\cdot|\alpha,\beta,\mu)\ast  f(\tau,\cdot)d\tau~\right \|_q \lesssim \|f\|_{L^\infty_tL^q_x([0,T]\times\BR^n)}.
			\end{eqnarray*}
		\end{theorem}
		
		\begin{proof}
			Application of Minkowski inequality for integrals (cf.~\cite[Theorem 2.4]{LiebLoss01}), followed by the definition of ${\bm M}_{\sigma,n}(t,x|\alpha,\beta,\mu)$ (see {\bf Notation \ref{MarquesNelsonKernels}}) gives, for values of $1\leq q<\infty$, the set of inequalities
			\begin{eqnarray*}
				\left\|~\displaystyle \int_0^t {\bm M}_{\sigma,n}(t-\tau,\cdot|\alpha,\beta,\mu)\ast f(\tau,\cdot)d\tau~\right \|_q \leq \\
				\leq \int_0^t \left\|~\displaystyle {\bm M}_{\sigma,n}(t-\tau,\cdot|\alpha,\beta,\mu)\ast f(\tau,\cdot)~\right \|_q d\tau  \\ \ \\
				\lesssim \begin{cases}
					\displaystyle 	\int_0^t  (t-\tau)^{\beta-1} \max_{\lambda \in \{\lambda_+,\lambda_-\}} \left\|~E_{\alpha,\beta}(-\lambda(-\Delta)^{\frac{\sigma}{2}} (t-\tau)^{\alpha})f(\tau,\cdot)~\right\|_q d\tau &,~\mu\neq 2 \\ \ \\
					\displaystyle \int_0^t  (t-\tau)^{\beta-1}~\left\|~E_{\alpha,\alpha}(-(-\Delta)^{\frac{\sigma}{2}} (t-\tau)^\alpha)f(\tau,\cdot)~\right\|_qd\tau &,~\mu= 2.
				\end{cases}
			\end{eqnarray*}	
			
			For $q=\infty$, the same set of inequalities holds from the norm identity
			$\displaystyle \left\|~\displaystyle \cdot ~\right \|_\infty =\lim_{q\rightarrow\infty}\left\|~\displaystyle \cdot~\right \|_q$
			and from a straightforward application of dominated convergence theorem (cf.~\cite[Theorem 1.8]{LiebLoss01}).
			
			Thereafter, under the conditions of {\bf Assumption \ref{AssumptionParameters}} one gets, by direct application of parts {\bf (1)} and {\bf (2)} of Corollary \ref{DispersiveCorollary} that
			\begin{eqnarray*}
				\left\|~\displaystyle \int_0^t {\bm M}_{\sigma,n}(t-\tau,\cdot|\alpha,\beta,\mu)\ast  f(\tau,\cdot)d\tau~\right \|_q 
				\lesssim \nonumber \\
				\lesssim	\int_0^t (t-\tau)^{\beta-1-\frac{\alpha}{\sigma}\left(\frac{n}{p}-\frac{n}{q}-\left(\gamma-\frac{\sigma}{\alpha}\cdot\beta\right)\right)}~
				\left\|~(-\Delta)^{\frac{\gamma}{2}-\frac{\sigma}{2\alpha}\cdot\beta} f(\tau,\cdot)~\right\|_{p}~ d\tau \\
				=\int_0^t (t-\tau)^{-\frac{\alpha}{\sigma}\left(\frac{n}{p}-\frac{n}{q}-\left(\gamma-\frac{\sigma}{\alpha}\right)\right)}~
				\left\|~(-\Delta)^{\frac{\gamma}{2}-\frac{\sigma}{2\alpha}\cdot\beta} f(\tau,\cdot)~\right\|_{p}~ d\tau
			\end{eqnarray*}
			hold for every $t\in [0,T]$.

			Next, from the combination of the conditions {\bf (I)}, {\bf (II)} and {\bf (III)} considered on {\bf Assumption \ref{AssumptionF}} for the function $f:[0,T]\times \mathbb{R}^n\longrightarrow \mathbb{R}$, we conclude that
			\begin{eqnarray}
				\label{MinkowskiIntIneq}
				\left\|~\displaystyle \int_0^t {\bm M}_{\sigma,n}(t-\tau,\cdot|\alpha,\beta,\mu)\ast  f(\tau,\cdot)d\tau~\right \|_q 
				\lesssim\Lambda(t)~\|f\|_{L^\infty_tL^q_x([0,T]\times\BR^n)},
			\end{eqnarray}
			with $$\displaystyle \Lambda(t):=\int_0^t (t-\tau)^{-\frac{\alpha}{\sigma}\left(\frac{n}{p}-\frac{n}{q}-\left(\gamma-\frac{\sigma}{\alpha}\right)\right)}~\tau^{\frac{\alpha}{\sigma}\left(\frac{n}{p}-\frac{n}{q}-\left(\gamma-\frac{\sigma}{\alpha}\right)\right)-1}~d\tau.$$
			
			We note that $\Lambda(t)$ can be determined with the aid of the Laplace convolution formula (\ref{LaplaceConvolution}), whereby the existence of the Laplace transforms
			\begin{center}
				$\displaystyle \mathcal{L}\left\{t^{-\eta}\right\}(s)$ and $\displaystyle\mathcal{L}\left\{t^{\eta-1}\right\}(s)$, with $\displaystyle \eta=\frac{\alpha}{\sigma}\left(\frac{n}{p}-\frac{n}{q}-\left(\gamma-\frac{\sigma}{\alpha}\right)\right)$,
			\end{center}
			follow from the condition
			\begin{eqnarray}
				\label{fDecayConstraint}		\displaystyle 0<\alpha-\frac{\alpha}{\sigma}\left(\frac{n}{p}-\frac{n}{q}-\left(\gamma-\frac{\sigma}{\alpha}\right)\right)< \alpha.
			\end{eqnarray}
			
			Indeed, condition {\bf (I)} of {\bf Assumption \ref{AssumptionF}} ensures that $f(\tau,\cdot)\in \dot{W}^{\gamma-\frac{\sigma}{\alpha},p}(\BR^n)$, since 
			$$\dot{W}^{\gamma-\frac{\sigma}{\alpha},p}(\BR^n)\cap \dot{W}^{\gamma-\frac{\sigma}{\alpha}\cdot\beta,p}(\BR^n)\subseteq\dot{W}^{\gamma-\frac{\sigma}{\alpha},p}(\BR^n).$$
			
			Then, by setting $\displaystyle \varepsilon:=\gamma-\frac{\sigma}{\alpha},$
			we recall that the set of parameter constraints, that results from parts {\bf (1)} and {\bf (2)} of Corollary \ref{DispersiveCorollary} for the substitution $\gamma\rightarrow \varepsilon$, gives rise to the following conditions:
			\begin{itemize}
				\item[\bf (i)] In case of $(\sigma,\gamma)\in \mathcal{S}_0\left(\alpha\right)~\wedge~(p,q)=(1,\infty)$, one has
				$$0<\frac{n}{p}-\frac{n}{q}-\varepsilon<\sigma.$$
				\item[\bf (ii)] In case of $(\sigma,\gamma)\in \mathcal{S}_{1,2}\left(\alpha\right)~\wedge~(p,q)\in \mathcal{R}_{1,2}\left(\sigma+\varepsilon\right)$, one has
				$$
				\frac{n}{2}-\varepsilon\leq \frac{n}{p}-\frac{n}{q}-\varepsilon<\min\left\{n-\varepsilon,\sigma\right\}.
				$$
				\item[\bf (iii)] In case of $(\sigma,\gamma)\in \mathcal{S}_{3}\left(\alpha\right)~\wedge~(p,q)\in \mathcal{R}_{3}\left(\varepsilon,\sigma+\varepsilon\right)$
				$$	0<\frac{n}{p}-\frac{n}{q}-\varepsilon<\min\left\{n-\varepsilon,\sigma\right\}. $$
			\end{itemize}
			
			Thus,
			$
			\displaystyle 0<\frac{\alpha}{\sigma}\left(\frac{n}{p}-\frac{n}{q}-\nu\right)<\alpha,
			$
			and hence, condition (\ref{fDecayConstraint}) is always satisfied. 
			Moreover, combination of the Laplace convolution formula (\ref{LaplaceConvolution}) with Laplace inversion gives
			\begin{eqnarray*}
				\Lambda(t)&=&\Gamma\left(1-\frac{\alpha}{\sigma}\left(\frac{n}{p}-\frac{n}{q}-\left(\gamma-\frac{\sigma}{\alpha}\right)\right)\right)\Gamma\left(\frac{\alpha}{\sigma}\left(\frac{n}{p}-\frac{n}{q}-\left(\gamma-\frac{\sigma}{\alpha}\right)\right)\right) \\ \ \\
				&=&\dfrac{\pi}{\sin\left(\frac{\alpha \pi}{\sigma}\left(\frac{n}{p}-\frac{n}{q}-\left(\gamma-\frac{\sigma}{\alpha}\right)\right) \right)},
			\end{eqnarray*}
			which is a positive constant independent of the time variable $t$. 
			
			Thus, for every $t\in [0,T]$ the inequality (\ref{MinkowskiIntIneq}) becomes
			\begin{eqnarray*}
				\left\|~\displaystyle \int_0^t {\bm M}_{\sigma,n}(t-\tau,\cdot|\alpha,\beta,\mu)\ast  f(\tau,\cdot)d\tau~\right \|_q \lesssim \|f\|_{L^\infty_tL^q_x([0,T]\times\BR^n)},
			\end{eqnarray*} as desired.
		\end{proof}
		
		\begin{theorem}\label{Theorem1Strichartz}
			Let $0<\alpha\leq \frac{1}{2}$, $2\alpha\leq \beta< 2\alpha+1$, $\mu>0$ and $\sigma>0$ be given.
			
			Assume further the following:
			\begin{itemize}
				\item {\bf Assumption \ref{AssumptionLambda}} holds for the constants $\lambda_+$ and $\lambda_-$ defined viz (\ref{RootsLambda});
				\item {\bf Assumption \ref{AssumptionParameters}} holds for the $4-$tuple $(\sigma,\gamma,p,q)$;
				\item {\bf Assumption \ref{AssumptionF}} holds for the function $f:[0,T]\times \BR^n\longrightarrow \BR$.
			\end{itemize}
			
			Then, for every datum
			$u_0\in \dot{W}^{\gamma-\frac{\sigma}{\alpha},p}(\BR^n)$ one has the following Strichartz estimates for the solution of the {\bf Cauchy Problem \ref{CP_parabolic}}, $u:[0,T]\times \BR^n\longrightarrow \BR$, obtained in Theorem \ref{Theorem1}:
			\begin{eqnarray*}
				\label{BoundATheorem1Strichartz}
				\|u\|_{L^s_tL^q_x([0,T]\times\BR^n)}&\lesssim& T^{\frac{1}{s}-\frac{\alpha}{\sigma}\left(\frac{n}{p}-\frac{n}{q}-\left(\gamma-\frac{\sigma}{\alpha}\right)\right)}~\left\|~(-\Delta)^{\frac{\gamma}{2}-\frac{\sigma}{2\alpha}} u_0~\right\|_{p} \\
				\nonumber &+&T^{\frac{1}{s}}~\|f\|_{L^\infty_tL^q_x([0,T]\times\BR^n)},
			\end{eqnarray*}
			only if 
			\begin{eqnarray}
				\label{MixedNormConstraintA} ~1\leq s<\frac{1}{\frac{\alpha}{\sigma}\left(\frac{n}{p}-\frac{n}{q}-\left(\gamma-\frac{\sigma}{\alpha}\right)\right)}.
			\end{eqnarray}
		\end{theorem}
		
		\begin{proof}
			Starting from the definition of ${\bm N}_{\sigma,n}(t,x|\alpha,\mu)$ (see {\bf Notation \ref{MarquesNelsonKernels}}), one has
			
			\begin{eqnarray*}
				\left\|~\displaystyle {\bm N}_{\sigma,n}(t,\cdot|\alpha,\mu)\ast u_0~\right \|_q \lesssim \\ \ \\
				\lesssim \begin{cases}
					\displaystyle 	 \max_{\lambda \in \{\lambda_+,\lambda_-\}} \left\|~E_{\alpha,1}\left(-\lambda(-\Delta)^{\frac{\sigma}{2}} t^\alpha\right)u_0~\right\|_q  &,~\mu\neq 2 \\ \ \\
					\displaystyle ~\left\|~E_{\alpha,1}\left(-(-\Delta)^{\frac{\sigma}{2}}t^\alpha\right) u_0~\right\|_q+t^\alpha \left\|~(-\Delta)^{\frac{\sigma}{2}}E_{\alpha,\alpha}\left(-(-\Delta)^{\frac{\sigma}{2}}t^\alpha\right) u_0~\right\|_q &,~\mu= 2.
				\end{cases}
			\end{eqnarray*}	
			
			Thereafter, under the {\bf Assumption \ref{AssumptionF}} for the constants $\lambda_+$ and $\lambda_{-}$ defined through eq. (\ref{RootsLambda}), the estimate
			\begin{center}
				$	\left\|~\displaystyle {\bm N}_{\sigma,n}(t,\cdot|\alpha,\mu)\ast u_0~\right \|_q \lesssim  t^{-\frac{\alpha}{\sigma}\left(\frac{n}{q}+\frac{n}{2}-\left(\gamma-\frac{\sigma}{\alpha}\right)\right)}~ \left\|~(-\Delta)^{\frac{\gamma}{2}-\frac{\sigma}{2\alpha}}u_0~\right\|_{p}$
			\end{center}
			follows by direct application of parts {\bf (1)} and {\bf (2)} of Corollary \ref{DispersiveCorollary} to the substitution $\gamma \rightarrow \gamma-\frac{\sigma}{\alpha}$. On the other hand, {\bf Assumption \ref{AssumptionF}} ensures that Theorem \ref{fBoundedProposition} can be applied to the function $f:[0,T]\times \BR^n\longrightarrow \BR$.
			
			Thus, for the solution of the {\bf Cauchy Problem \ref{CP_parabolic}}, determined in Theorem \ref{Theorem1}, we end up with the following dispersive estimate
			\begin{center}
				$\| u(t,\cdot)\|_q \lesssim t^{-\frac{\alpha}{\sigma}\left(\frac{n}{p}-\frac{n}{q}-\left(\gamma-\frac{\sigma}{\alpha}\right)\right)}~ \left\|~(-\Delta)^{\frac{\gamma}{2}-\frac{\sigma}{2\alpha}}u_0~\right\|_{p}+\|f\|_{L^\infty_tL^q_x([0,T]\times\BR^n)}$.
			\end{center}
			
			Moreover, by adopting the mixed-norm (\ref{mixedNorm}) of $L^s_tL^q_x([0,T]\times\BR^n)$, the sequence of inequalities
			\begin{eqnarray*}
				\|u\|_{L^s_tL^q_x([0,T]\times\BR^n)}&\leq& \left(\int_0^T t^{-s\cdot\frac{\alpha}{\sigma}\left(\frac{n}{p}-\frac{n}{q}-\left(\gamma-\frac{\sigma}{\alpha}\right)\right)} \|~(-\Delta)^{\frac{\gamma}{2}-\frac{\sigma}{2\alpha}}u_0~\|_{2}^s~dt\right)^{\frac{1}{s}} \\
				&+& \left(\int_0^T \|f\|_{L^\infty_tL^q_x([0,T]\times\BR^n)^s}^sdt\right)^{\frac{1}{s}} \\
				&\lesssim& T^{\frac{1}{s}-\frac{\alpha}{\sigma}\left(\frac{n}{p}-\frac{n}{q}-\left(\gamma-\frac{\sigma}{\alpha}\right)\right)} \left\|~(-\Delta)^{\frac{\gamma}{2}-\frac{\sigma}{2\alpha}}u_0~\right\|_{p}+T^{\frac{1}{s}}\|f\|_{L^\infty_tL^q_x([0,T]\times\BR^n)} 
			\end{eqnarray*}
			results from the Minkowski inequality and by standard integration techniques, whereby the condition (\ref{MixedNormConstraintA}) is to ensure the monotonicity of the norm of $L^s_tL^q_x([0,T]\times\BR^n)$ and the equality 
			\begin{eqnarray}
				\label{StrichartzIntegralA}\displaystyle \int_0^T t^{-s\cdot \frac{\alpha}{\sigma}\left(\frac{n}{p}-\frac{n}{q}-\left(\gamma-\frac{\sigma}{\alpha}\right)\right)}dt=\frac{T^{1-s\cdot \frac{\alpha}{\sigma}\left(\frac{n}{p}-\frac{n}{q}-\left(\gamma-\frac{\sigma}{\alpha}\right)\right)}}{1-s\cdot\frac{\alpha}{\sigma}\left(\frac{n}{p}-\frac{n}{q}-\left(\gamma-\frac{\sigma}{\alpha}\right)\right)}.
			\end{eqnarray}
			
			This concludes the proof of Theorem \ref{Theorem1Strichartz}.
		\end{proof}
		\begin{theorem}\label{Theorem2Strichartz}
			Let $0<\alpha\leq \frac{1}{2}$, $2\alpha\leq \beta< 2\alpha+1$, $\mu>0$ and $\sigma>0$ be given.
			
			Assume further the following:
			\begin{itemize}
				\item {\bf Assumption \ref{AssumptionLambda}} holds for the constants $\lambda_+$ and $\lambda_-$ defined viz (\ref{RootsLambda});
				\item {\bf Assumption \ref{AssumptionParameters}} holds for the $4-$tuple $(\sigma,\gamma,p,q)$;
				\item {\bf Assumption \ref{AssumptionF}} holds for the function $f:[0,T]\times \BR^n\longrightarrow \BR$.
			\end{itemize}
			
			Then, for every data $(u_0,u_1)\in \dot{W}^{\gamma,p}(\BR^n)\times\dot{W}^{\gamma-\frac{\sigma}{\alpha},p}(\BR^n)$, the solution of the {\bf Cauchy Problem \ref{CP_hyperbolic}}, $u:[0,T]\times \BR^n\longrightarrow \BR$, obtained in Theorem \ref{Theorem2} satisfies the Strichartz estimate
			\begin{eqnarray*}
				\|u\|_{L^s_tL^q_x\left([0,T]\times\BR^n\right)}&\lesssim& T^{\frac{1}{s}-\frac{\alpha}{\sigma}\left(\frac{n}{q}+\frac{n}{2}-\gamma \right)}
				\left(~\|~(-\Delta)^{\frac{\gamma}{2}} u_0~\|_{p}+\|~(-\Delta)^{\frac{\gamma}{2}-\frac{\sigma}{2\alpha}} u_1~\|_{p}~\right)\\
				&+& T^{\frac{1}{s}}~\|f\|_{L^\infty_tL^q_x\left([0,T]\times\BR^n\right)}
			\end{eqnarray*}
			only if
			\begin{eqnarray}
				\label{MixedNormConstraintB}
				1\leq s<\frac{1}{\frac{\alpha}{\sigma}\left(\frac{n}{p}-\frac{n}{q}-\gamma\right)}.
			\end{eqnarray}
		\end{theorem}
		
		\begin{proof}
			From the proof of Theorem \ref{Theorem1Strichartz} we can immediately infer that for the auxiliar function $v:[0,T]\times \BR^n\rightarrow\BR$, defined viz
			\begin{eqnarray*}
				v(t,x):=\left({\bm N}_{\sigma,n}(t,\cdot|\alpha,\mu)*u_0\right)(x)+\int_{0}^t \left[~{\bm M}_{\sigma,n}(t-\tau,\cdot|\alpha,\beta,\mu)\ast f(\tau,\cdot)~\right](x)~d\tau,
			\end{eqnarray*} the following Strichartz estimate
			\begin{eqnarray*}
				\displaystyle \|v\|_{L^s_tL^q_x([0,T]\times\BR^n)}
				&\lesssim& T^{\frac{1}{s}-\frac{\alpha}{\sigma}\left(\frac{n}{p}-\frac{n}{q}-\gamma\right)} \|~(-\Delta)^{\frac{\gamma}{2}}u_0~\|_{p}+T^{\frac{1}{s}}\|f\|_{L^\infty_tL^q_x([0,T]\times\BR^n)} 
			\end{eqnarray*} holds true for values of $\frac{1}{2}<\alpha\leq 1$. 
			Thus, it remains to show that 
			\begin{eqnarray}
				\label{StrichartzJorge}
				\left(\int_0^T\left\|~\displaystyle {\bm J}_{\sigma,n}(t,\cdot|\alpha,\mu)\ast u_1~\right \|_q^s dt\right)^{\frac{1}{s}}  \lesssim T^{\frac{1}{s}-\frac{\alpha}{\sigma}\left(\frac{n}{p}-\frac{n}{q}-\gamma\right)}~\|~(-\Delta)^{\frac{\gamma}{2}-\frac{\sigma}{2\alpha}} u_1~\|_{p},
			\end{eqnarray}
			holds for every $u_1\in \dot{W}^{\gamma-\frac{\sigma}{\alpha},p}(\BR^n)$.
			
			Indeed, from the definition of ${\bm J}_{\sigma,n}(t,x|\alpha,\mu)$ provided by {\bf Notation \ref{JorgeKernel}}, one has
			\begin{eqnarray*}
				\left\|~\displaystyle {\bm J}_{\sigma,n}(t,\cdot|\alpha,\mu)\ast u_1~\right \|_q \lesssim \\ \ \\
				\lesssim \begin{cases}
					\displaystyle 	 \max_{\lambda \in \{\lambda_+,\lambda_-\}} t~ \left\|~E_{\alpha,2}\left(-\lambda(-\Delta)^{\frac{\sigma}{2}} t^\alpha\right)u_1~\right\|_q  &,~\mu\neq 2 \\ \ \\
					\displaystyle ~t\left\|~E_{\alpha,1}\left(-(-\Delta)^{\frac{\sigma}{2}}t^\alpha\right) u_1~\right\|_q+t \left\|~E_{\alpha,2}\left(-(-\Delta)^{\frac{\sigma}{2}}t^\alpha\right) u_1~\right\|_q &,~\mu= 2.
				\end{cases}
			\end{eqnarray*}	
			
			Then, following {\it mutatis mutandis} the scheme of proof employed in Theorem \ref{fBoundedProposition}, the estimate
			\begin{center}
				$	\left\|~\displaystyle {\bm J}_{\sigma,n}(t,\cdot|\alpha,\mu)\ast u_1~\right \|_q \lesssim  t^{-\frac{\alpha}{\sigma}\left(\frac{n}{p}-\frac{n}{p}-\gamma\right)}~ \|~(-\Delta)^{\frac{\gamma}{2}-\frac{\sigma}{2\alpha}} u_1~\|_{2}$
			\end{center}
			results by an algebraic manipulation, based on the equality $$\displaystyle t^{-\frac{\alpha}{\sigma}\left(\frac{n}{q}+\frac{n}{2}-\gamma\right)}=t\times t^{-\frac{\alpha}{\sigma}\left(\frac{n}{q}+\frac{n}{2}-\left(\gamma-\frac{\sigma}{\alpha}\right)\right)}.$$ 
			
			Hence, the estimate (\ref{StrichartzJorge}) follows from the condition (\ref{MixedNormConstraintB}).
		\end{proof}
		\newpage 
		\begin{remark}
			The well-posedness of {\bf Cauchy Problem \ref{CP_parabolic}} \& {\bf Cauchy Problem \ref{CP_hyperbolic}}, provided by Theorem \ref{Theorem1Strichartz} and Theorem \ref{Theorem2Strichartz} respectively, 
			is ensured by the combination dispersive estimates obtained in Corollary \ref{DispersiveCorollary} with the retarded Strichartz estimate, obtained the Theorem \ref{fBoundedProposition}.
		\end{remark}
		
		
		\begin{figure}
			\tikzstyle{level 1}=[level distance=18mm, sibling distance=40mm]
			\tikzstyle{level 2}=[level distance=22mm, sibling distance=30mm]
			\tikzstyle{level 3}=[level distance=28mm]
			\begin{tikzpicture}[grow=right,->,>=angle 60]
				\node {{\bf Assumptions \ref{AssumptionLambda} \& \ref{AssumptionParameters}}}
				child {node {{$\frac{1}{2}<\alpha\leq 1$}}
					child {node {{$f$}}
						child[-] {node{{\bf Assumption \ref{AssumptionF}}}}  
					}
					child {node{{$(u_0,u_1)$}}
						child[-] {node{{$\dot{W}^{\gamma,p}\times \dot{W}^{\gamma-{\frac{\sigma}{\alpha}},p}$} data}}  
					}
				}
				child {node {{$0<\alpha\leq \frac{1}{2}$}}
					child {node{{$f$}}
						child[-] {node{{\bf Assumption \ref{AssumptionF}}}}  
					}
					child {node{{$u_0$}}
						child[-] {node{{$\dot{W}^{\gamma-\frac{\sigma}{\alpha},p}$} datum}}  
					}
				};
			\end{tikzpicture}
			\caption{Binary tree diagram sketch of the set of main assumptions on required on {\bf Theorem \ref{Theorem1Strichartz}} and {\bf Theorem \ref{Theorem2Strichartz}}.}
		\end{figure}
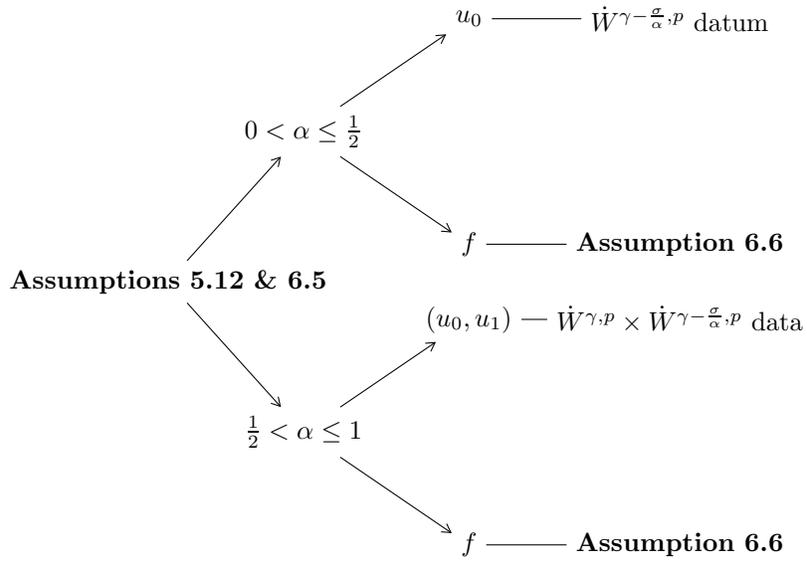
		
		
		\begin{figure}
			\begin{tikzpicture}[scale=8]
				\tkzInit[xmin=0, xmax=.70, ymin=0, ymax=.70]
				\tkzDrawX[noticks, label=\(\sigma\)]
				\tkzDrawY[noticks, label=\( \gamma\)]
				\tkzDefPoints{0/0/O}\tkzLabelPoints[below left](O)
				\tkzDefPoints{	0.5/0/A,
					0.75/.5/B,
					0.75/1/C,
					0.5/0.75/D,
					0/.5/E,
					.5/.5/F
				}
				
				\tkzDefPoint(0,1){M}
				\tkzDefPoint(0.75,0){Q}
				
				\tkzDefPoint(0.5,0.5){N}
				
				\tkzDefPoint(1,0.5){B'}
				\tkzDefPoint(1,1){N'}
				\tkzDefPoint(1,0){P'}
				
				\tkzDefPointWith[colinear=at M](O,C)\tkzGetPoint{aux}

				
				\tkzInterLL(M,aux)(B,C)\tkzGetPoint{M'}
				\tkzInterLL(M,M')(A,D)\tkzGetPoint{I2}		
				
				\tkzInterLL(O,C)(M,P')\tkzGetPoint{I3}
				\tkzInterLL(O,C)(E,N)\tkzGetPoint{I4}
				
				\tkzInterLL(O,N)(E,A)\tkzGetPoint{I5}
				\tkzInterLL(O,B')(E,A)\tkzGetPoint{I6}
				
				\tkzPointShowCoord[xlabel=\(n\alpha\), ylabel=\(n\), xstyle={below=2mm},ystyle={left=1mm}](C)


				
				\tkzPointShowCoord[xlabel=\(\frac{n\alpha}{1+\alpha}\),ylabel=\(\frac{n}{1+\alpha}\), xstyle={below=2mm},ystyle={left=1mm}](I3)
				
				\tkzPointShowCoord[xlabel=\(n\), xstyle={below=2mm},ystyle={left=1mm}](P')
				
				\tkzDrawSegments(I3,C)
				\tkzDrawSegments[dotted](I3,M)
				
				
				\tkzDrawPoints[fill=white](M,C,I3)
				

				\begin{scope}[on background layer]
					\tkzFillPolygon[gray!70](I3,C,M)
				\end{scope}
			\end{tikzpicture}
			\caption{Region $\mathcal{S}_{0}(\alpha)$ ($\frac{1}{2}<\alpha<1$) adopted on {\bf Assumption \ref{AssumptionParameters}}.}\label{StrichartzFigure1} 
		\end{figure}
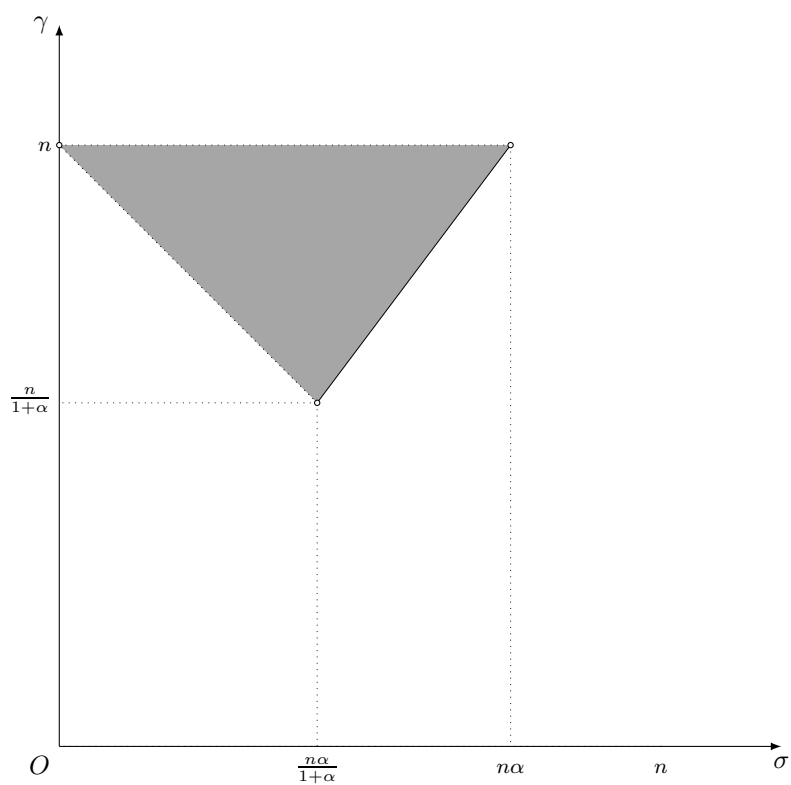

		\begin{figure}
			\begin{tikzpicture}[scale=8]
				\tkzInit[xmin=0, xmax=.70, ymin=0, ymax=.70]
				\tkzDrawX[noticks, label=\(\sigma\)]
				\tkzDrawY[noticks, label=\( \gamma\)]
				\tkzDefPoints{0/0/O}\tkzLabelPoints[below left](O)
				\tkzDefPoints{	0.5/0/A,
					0.75/.5/B,
					0.75/1/C,
					0.5/0.75/D,
					0/.5/E,
					.5/.5/F
				}
				
				\tkzDefPoint(0,1){M}
				\tkzDefPoint(0.75,0){Q}
				
				\tkzDefPoint(0.5,0.5){N}
				
				\tkzDefPoint(1,0.5){B'}
				\tkzDefPoint(1,1){N'}
				\tkzDefPoint(1,0){P'}
				
				\tkzDefPointWith[colinear=at M](O,C)\tkzGetPoint{aux}

				
				\tkzInterLL(M,aux)(B,C)\tkzGetPoint{M'}
				\tkzInterLL(M,M')(A,D)\tkzGetPoint{I2}		
				
				\tkzInterLL(O,C)(E,A)\tkzGetPoint{I3}
				\tkzInterLL(O,C)(E,N)\tkzGetPoint{I4}
				
				\tkzInterLL(O,N)(E,A)\tkzGetPoint{I5}
				\tkzInterLL(O,B')(E,A)\tkzGetPoint{I6}
				
				\tkzPointShowCoord[xlabel=\(n\alpha\), ylabel=\(n\), xstyle={below=2mm},ystyle={left=1mm}](C)

				\tkzPointShowCoord[noxdraw,noydraw,xlabel=\(\frac{n}{2}\), ylabel=\(\frac{n}{2}\), xstyle={below=2mm},ystyle={left=1mm}](N)

				\tkzPointShowCoord[xlabel=\(\frac{n\alpha}{2(1+\alpha)}\), ylabel=\(\frac{n}{2(1+\alpha)}\), xstyle={below=2mm},ystyle={left=1mm}](I3)
				
				\tkzPointShowCoord[xlabel=\(\frac{n\alpha}{2}\), xstyle={below=2mm},ystyle={left=1mm}](I4)
				
				\tkzPointShowCoord[xlabel=\(n\), xstyle={below=2mm},ystyle={left=1mm}](P')
				
				\tkzDrawSegments(I3,C)
				\tkzDrawSegment[dotted](M,N')
				\tkzDrawSegment[dotted](P',N')
				\tkzDrawSegment[dotted](E,I3)
				
				\tkzDrawSegment[dotted](Q,C)
				
				\tkzDrawPoints[fill=white](M,C,I3,E)
				

				\begin{scope}[on background layer]
					\tkzFillPolygon[gray!70](I3,I4,E)
					\tkzFillPolygon[black!20!white](E,I4,C,M)
				\end{scope}
			\end{tikzpicture}
			\caption{Regions $\mathcal{S}_{1,2}(\alpha)$ and $\mathcal{S}_{3}(\alpha)$ ($\frac{1}{2}<\alpha<1$) adopted on {\bf Assumption \ref{AssumptionParameters}}. $\mathcal{S}_{1,2}(\alpha)$ corresponds to the dark grey region whereas $\mathcal{S}_3(\alpha)$ corresponds to union of the light grey region with the straight line $\displaystyle \left\{\left(\sigma,\frac{n}{2}\right)~:~0<\sigma\leq\frac{n\alpha}{2}\right\}.$}\label{StrichartzFigure2} 
		\end{figure}
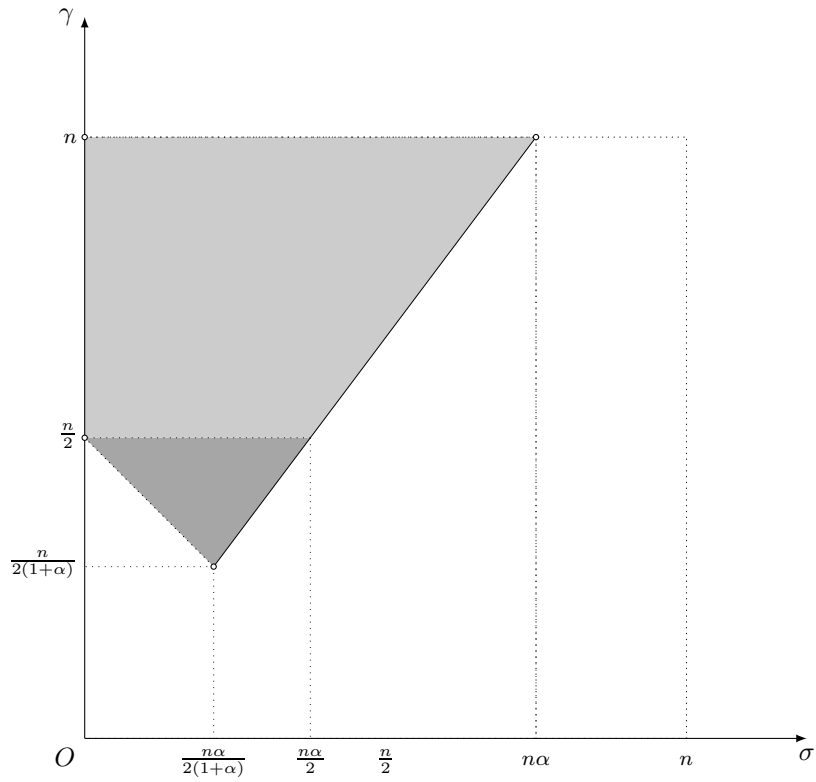

		\newpage
		~
		\newpage
		\section{Conclusions}\label{Conclusions}
		
		\subsection{Summary of results}
		
		We have investigated questions related to dispersive and Strichartz estimates for the {\bf Cauchy Problem \ref{CP_parabolic}} and the {\bf Cauchy Problem \ref{CP_hyperbolic}} from their closed-form representations.
		
		The framework developed in Section \ref{DispersiveSection} constitute a significant part of the paper. Along this we have pinned a general scheme to achieve the optimality of the decay estimates, independently of the dimension of the ambient space $\BR^n$. As a consequence, the main estimates collected in Corollary \ref{DispersiveCorollary}, permits us to treat $L^p-L^q$ and $L^p-\dot{W}^{\sigma,q}$ decay estimates, at least, in lower dimensions with respect to $\sigma$ (see also Proposition \ref{YoungIneqProposition}). That also includes $L^1-L^\infty$ and $L^1-\dot{W}^{\sigma,\infty}$ decay estimates, although we are unable to claim the sharpness of these estimates, for the reasons explained throughout Subsection \ref{ProofStrategy}.
		
		Note also that in case where such estimates are not possible to reach -- such as the wave equation with structural damping highlighted on the beginning of the paper \cite{CCD08} -- our approach offers the possibility to look instead for $\dot{W}^{\gamma,p}-L^q$ and $\dot{W}^{\gamma,p}-\dot{W}^{\sigma,q}$ decay estimates, whose datum is a function with membership on the {\it homogeneous Sobolev space} $\dot{W}^{\gamma,p}(\BR^n)$ (see also Proposition \ref{SobolevYoungProposition}).
		
		In Section \ref{MainSection}, we have used the estimates obtained in Section \ref{DispersiveSection} to prove first, in Theorem \ref{fBoundedProposition}, a {retarded Strichartz estimate} for the inhomogeneous part of the {\bf Cauchy Problem \ref{CP_parabolic}} \& {\bf Cauchy Problem \ref{CP_hyperbolic}}.  Afterwards, by combining Corollary \ref{DispersiveCorollary} and Theorem \ref{fBoundedProposition}, we were able to prove, in {Theorem \ref{Theorem1Strichartz}} and {Theorem \ref{Theorem2Strichartz}, a primer version of Strichartz estimates to both Cauchy problems.
			
			
			
			Regarding {Theorem \ref{fBoundedProposition}}, it is worth to observe that its proof turned out to be rather technical. Arguably, we have adopted in {\bf Assumption \ref{AssumptionF}} the most strongest regularity and decay properties on  $\displaystyle f:[0,T]\times \BR^n\rightarrow\BR$ with the aim achieve a similar version of the so-called {\it Christ-Kiselev lemma} (cf.~\cite[Lemma 2.4]{Tao06}). Amongst the many still open questions, obtaining a well-adapted version of \cite[THEOREM 1.2.]{KTao98}~to our case is indeed an interesting issue to be deepened in forthcoming research papers.

			In comparison with other approaches, our major contribution stems, firstly, on the possibility of proving decay estimates for both Cauchy problems without studying the asymptotic behaviour of the fundamental solution. In particular, we note that the fundamental solutions encoded by higher dimensional fractional differential equations -- that is, the convolution kernels that yield from the {\it Fourier convolution formula} -- are mostly Fox-H functions (cf.~\cite[Subsection 1.12]{KST06}). And, as it was highlighted on the paper \cite{KSZ17}, obtaining optimal decay estimates for it revealed to be a non-trivial task.
			
			Also, we notice that the study of the size and the decay of the solutions of Cauchy problems of heat type, possessing memory terms, has been extensively studied during the last years (see e.g. \cite{KSVZ16,CQW21} and the references given there). And prior to the study of dispersive estimates, there seems that the study of Strichartz estimates for the solutions of Cauchy problems, similar to {\bf Cauchy Problem \ref{CP_parabolic}} \& {\bf Cauchy Problem \ref{CP_hyperbolic}}, has not yet been fully addressed, up to authors's knowledge. The present paper aims also to fill this gap.
			\newpage

			\subsection{Further remarks and open problems}
			
			The wide class of space-time-fractional operators, considered on the formulation of {\bf Cauchy Problem \ref{CP_parabolic}} \& {\bf Cauchy Problem \ref{CP_hyperbolic}}, correspond to a borderline case of the following operator:
			\begin{eqnarray}
				\label{SpaceTimeFractional}	\partial_t^{\gamma}+\mu (-\Delta)^\theta \partial_t^\alpha +(-\Delta)^\sigma,&\mbox{with}~0<\theta\leq \sigma,~~ 0<\alpha\leq 1,~~ \alpha<\gamma \leq 2. 
			\end{eqnarray}
			
			For such general class of operators, Luchko-Gorenflo's operational approach considered in \cite{luchko1999operational} provides us a general framework to determine, in a similar fashion, closed-form representations to the ones obtained in Theorem \ref{SolutionCP_parabolic} and Theorem \ref{SolutionCP_hyperbolic}. On this case, the convolution kernels turned out be described in terms of {\it multi-index Mittag-Leffler functions} (see~\cite[Theorem 4.1]{luchko1999operational} and \cite[eq.~(46)]{luchko1999operational}). 
			
			At this point, it should be stressed that the choice of $\gamma=2\alpha$ on eq. (\ref{SpaceTimeFractional}), already considered e.g. in \cite{OrsingherBeghin04,LVaz20} for the time-fractional telegraph equation, provides us a fairly accurate time-fractional counterpart of the $\sigma-$models investigated in depth on the series of papers \cite{PMR15,AE17,AE_21,AE22}. Up to our knowledge, for such class of multi-index functions the monotonicity and decay properties are not very well-known on the literature (cf.~\cite[Chapter 6]{GKMaiR14}) so that one cannot embody it in our scheme of proof to obtain dispersive estimates and to plan tangible applications towards relaxation theory as well, as tactically explained in Section \ref{Discussion}.
			Therefore, for the general case the following interesting open questions remain:
			\begin{enumerate}
				\item[\bf Q1:] Is it possible to achieve endpoint decay estimates for the {\it multi-index Mittag-Leffler functions}? At least, using an asymptotic scheme similar to the one considered in \cite{KSZ17}?
				\item[\bf Q2:] Under which choices of $\gamma$ and $\alpha$ are possible to achieve dispersive estimates, similar to those obtained in Section \ref{DispersiveSection}?
			\end{enumerate}
			
			For the choice $\gamma=2\alpha$, it seems likely that our results can be extended straightforwardly for general choices of $0<\theta\leq \sigma$, since with the aid of eq.~(\ref{LaplaceIdentityMittagLeffler}) we are able to say, from its solution representation, that one needs to study, in most of the cases, the decay of the Fourier multipliers of the type 
			\begin{eqnarray*}
				\displaystyle \frac{1}{\lambda_+(\xi)-\lambda_-(\xi)}~\frac{E_{\alpha,1}\left(-|\xi|^{\sigma}\lambda_\pm(\xi)~t^{\alpha}\right)}{\lambda_\pm(\xi)},& \displaystyle \frac{\lambda_\pm(\xi)~t^{\beta-1}~E_{\alpha,\beta}\left(-|\xi|^{\sigma}\lambda_\pm(\xi)~t^{\alpha}\right) }{\lambda_+(\xi)-\lambda_-(\xi)},
			\end{eqnarray*} ($\lambda_+(\xi)\neq \lambda_-(\xi)$) and $t^{\beta-1}E_{\alpha,\beta}^2(-|\xi|^\sigma t^\alpha)$ ($\lambda_+(\xi)= \lambda_-(\xi)$), whereby
			\begin{eqnarray*}
				\lambda_{\pm}(\xi) =\begin{cases}
					\displaystyle \frac{\mu}{2}|\xi|^{2\theta-\sigma}\pm i\sqrt{1-\frac{\mu^2}{4}|\xi|^{4\theta-2\sigma}}&, 0< \mu<2|\xi|^{\sigma-2\theta} \\ \ \\
					\displaystyle \frac{\mu}{2}|\xi|^{2\theta-\sigma}\pm \sqrt{\frac{\mu^2}{4}|\xi|^{4\theta-2\sigma}-1} &, \mu\geq 2|\xi|^{\sigma-2\theta}. 
				\end{cases}
			\end{eqnarray*}

			In case where $\theta=\dfrac{\sigma}{2}$ we are in conditions to claim that our results are enough to cover the so-called critical case (i.e. when $\lambda_\pm(\xi)$ are constants), but we cannot yet claim it for the remaining cases 
			$\displaystyle 0<\theta<\frac{\sigma}{2}$ and $\displaystyle \frac{\sigma}{2}<\theta\leq \sigma$,
			already treated by Pham et al. in \cite{PMR15} for values of $\alpha=\beta=1$. The main gap to be to be circumvented is the fact that the terms
			\begin{eqnarray*}
				\dfrac{1}{\lambda_+(\xi)-\lambda_-(\xi)} &\mbox{and}&
				\dfrac{\lambda_\pm(\xi)}{\lambda_+(\xi)-\lambda_-(\xi)}
			\end{eqnarray*}
			and unbounded for values of $\displaystyle  |\xi|^{\sigma-2\theta}\approx\frac{\mu}{2}$. 
			However, based on the identity $\lambda_{+}(\xi)\cdot \lambda_-(\xi)=1$, the following question comes naturally:
			\begin{enumerate}
				\item[\bf Q3:] Can we also obtain estimates, similar to the ones reached in Section \ref{DispersiveSection}, to regions defined implicitly by 
				\begin{center}
					$\displaystyle |\lambda_{+}(\xi)-\lambda_{-}(\xi)|\geq  \varepsilon\cdot \max\left\{ \mu,\mu |\xi|^{2\theta-\sigma}\right\}$,
					for some $\varepsilon>0$?
				\end{center}
			\end{enumerate}
			
			Broadly speaking, we essentially conjecture a perturbation argument, underlying to the asymptotic behaviour of $\varepsilon:=\varepsilon(\tau)$ on small scales $0<\tau<1$: 
			\begin{center}
				$\displaystyle \varepsilon \sim \tan\left(\frac{\pi\tau}{2}\right)$ $\left(~0< \mu<2|\xi|^{\sigma-2\theta}~\right)$~~~$\wedge$~~~$\displaystyle \varepsilon \sim \tanh\left(\tau\right)$ $\left(~\mu>2|\xi|^{\sigma-2\theta}~\right)$.
			\end{center}
			
			From the background of semigroup theory, we already know from the abstract results obtained in \cite[Section VI.3.a]{engel2000one} that the application of perturbation arguments, similar to the above one, are mostly feasible in case of $\displaystyle \frac{\sigma}{2}<\theta\leq\sigma$ -- that is, for the case where the elastic operator is bounded above by the damping operator. But due to the fact that the generalizations of the framework considered in \cite{PMR15}, on which the authors already treat the case $\displaystyle \frac{\sigma}{2}<\theta\leq \sigma$, requires in our case to handle, at most, with subordination formulae for the generalized Mittag-Leffler functions $E_{\alpha,\beta}$ in terms of probability density functions defined through the so-called Wright functions (cf.~\cite{Mainardi2020}), we decided to postpone the treatment of the aforementioned remaining cases for a forthcoming paper.
			
			We conclude this section by noticing that Cordero-Zucco's approach on Strichartz estimates for the vibrating plate equation on Sobolev and modulation spaces (case of $\mu\rightarrow 0^+$; $\sigma=2$) (see~\cite{CZ11,CZ11modulation}) also falls into the framework of the class of function spaces and pseudo-differential operators investigated in depth in Section \ref{DispersiveSection} and Section \ref{MainSection}. This is also an interesting issue to be deepened in a nearly future, as well as the following broadly general questions:
			\begin{enumerate}
				\item[\bf Q4:] Employing the well-established {Littlewood-Paley's theory}, can we relax some of the results enclosed on this paper?
				\item[\bf Q5:] Can we replace the {\it homogeneous Sobolev spaces}, adopted throughout this paper, by homogeneous analogues of Besov spaces?
			\end{enumerate}
			
			\section{Acknowledgments}
			\href{http://orcid.org/0000-0002-9117-2021}{Nelson~Faustino} was supported by The Center for Research and Development in Mathematics and Applications (CIDMA) through the Portuguese Foundation for Science and Technology (FCT), references UIDB/04106/2020 and UIDP/04106/2020; \href{https://orcid.org/0000-0003-3825-6389}{Jorge~Marques} was supported by Centre for Business and Economics Research (CeBER) through the Portuguese Foundation for Science and Technology (FCT), reference UIDB/05037/2020.
			
			The authors would like to thank Antero Neves for his precious advice and help on creating TikZ code used to build some of the graphics throughout the paper. 
			Last but not least, the authors would like to thank the reviewers for carefully reading the preliminary version of the manuscript. Their remarks and suggestions allowed us to clarify several issues as well as improve some of the main results.

			\newpage 
			\appendix

			\section{Proof of Lemma \ref{LrMittagLeffler}}\label{AppendixA}
			
			\begin{proof}
				First, we recall that for $\beta=1$, one has
				$\displaystyle \frac{\Gamma(1)}{\Gamma(1+\alpha)}=\frac{1}{\Gamma(1+\alpha)}$.
				Thus, application of Theorem \ref{EstimatesMittagLeffler} \& Theorem \ref{OptimalEstimatesMittagLeffler} gives rise to  
				\begin{eqnarray}
					\label{IntegralIneq}
					\int_0^\infty\left|~\Gamma(\beta) \rho^{\eta} E_{\alpha,\beta}(-\lambda \rho^\sigma t^{\alpha})~\right|^r\rho^{n-1}d\rho= 	\int_0^\infty\left|~\Gamma(\beta) E_{\alpha,\beta}(-\lambda \rho^\sigma t^{\alpha})~\right|^r\rho^{n+r\eta-1}d\rho  \nonumber \\ \nonumber \\  \leq \begin{cases} \displaystyle \int_0^\infty \frac{\rho^{n+r\eta-1}}{\left(1+\sqrt{\frac{\Gamma(1+\alpha)}{\Gamma(1+2\alpha)}}~\lambda \rho^\sigma t^{\alpha}\right)^{2r}} d\rho &, \lambda\geq 0;~\beta=\alpha
						\\ \ \\
						\displaystyle \int_0^\infty \frac{\rho^{n+r\eta-1}}{\left(1+\frac{\Gamma(\beta)}{\Gamma(\beta+\alpha)}\lambda \rho^\sigma t^{\alpha}\right)^r} d\rho&, \lambda\geq 0;~\beta\in \{1\}\cup (\alpha,+\infty) \\ \ \\
						\displaystyle \int_0^\infty\frac{C^r|\Gamma(\beta)|^r\rho^{n+r\eta-1}}{(1+|\lambda| \rho^\sigma t^{\alpha})^r} d\rho&, \frac{\pi\alpha}{2}<\theta<\pi \alpha;~\theta\leq |\arg(\lambda)|\leq \pi,
					\end{cases}
				\end{eqnarray}
				where $C>0$ is the constant that results from Theorem \ref{EstimatesMittagLeffler}.
				
				For the case of $\eta=-s$ and $a=r$ resp. $a=2r$, the set of constraints {\bf (i)}, {\bf (ii)} and {\bf (iii)}, appearing on the statement of 
				Lemma \ref{LrMittagLeffler}, assures that
				\begin{center}
					$\displaystyle \Re(r)>\Re\left(\frac{n-rs}{\sigma}\right)>0$ resp. $\displaystyle \Re(2r)>\Re\left(\frac{n-rs}{\sigma}\right)>0$  
				\end{center}
				so that one can apply the Mellin integral identity (\ref{MellinId}) to simplify the previous set of inequalities, appearing on the right-hand side of (\ref{IntegralIneq}). Namely, one has 
				\begin{eqnarray*}	\left|~\Gamma(\beta)~\right|^r\int_0^\infty\left|~ E_{\alpha,\beta}(-\lambda \rho^\sigma t^{\alpha})~\right|^r\rho^{n-rs-1}d\rho \leq C_{r,\sigma,n}^{(s)}(\alpha,\beta,\lambda) t^{-\frac{\alpha}{\sigma}\left({n-rs}\right)}
				\end{eqnarray*}
				\begin{eqnarray*}
					=\begin{cases} \displaystyle \frac{\Gamma\left(\frac{ n-rs}{\sigma}\right)\Gamma\left(\frac{r(2\sigma+s)-n}{\sigma}\right)}{\sigma~\Gamma(2r)}\left(\sqrt{\frac{\Gamma(1+\alpha)}{\Gamma(1+2\alpha)}}~\lambda  t^{\alpha}\right)^{-\frac{n-rs}{\sigma}} \displaystyle,~~~~\lambda\geq 0~;~\beta=\alpha
						\\ \ \\
						\displaystyle \frac{\Gamma\left(\frac{ n-rs}{\sigma}\right)\Gamma\left(\frac{r(\sigma+s)-n}{\sigma}\right)}{\sigma~\Gamma(r)}~\left(\frac{\Gamma(\beta)}{\Gamma(\beta+\alpha)}~\lambda t^{\alpha}\right)^{-\frac{n-rs}{\sigma}} ~~~~,\lambda\geq 0~;~\beta\in \{1\}\cup (\alpha,+\infty) \\ \ \\
						\displaystyle \frac{C^r~|\Gamma(\beta)|^r\Gamma\left(\frac{ n-rs}{\sigma}\right)\Gamma\left(\frac{r(\sigma+s)-n}{\sigma}\right)}{\sigma~\Gamma(r)}\left(~|\lambda| t^{\alpha}\right)^{-\frac{n-rs}{\sigma}},~\frac{\pi\alpha}{2}<\theta<\pi \alpha~~;~~\theta\leq |\arg(\lambda)|\leq \pi,
					\end{cases}
				\end{eqnarray*}
				where $C_{r,\sigma,n}^{(s)}(\alpha,\beta,\lambda)$ is the constant defined on the statement of Lemma \ref{LrMittagLeffler}.
				
				Hence, from the monotonicity of the norm underlying to the weighted Lebesgue space $L^r((0,\infty),\rho^{n-1}d\rho)$, there holds that the previous inequality is equivalent to (\ref{IneqC}), as desired.
			\end{proof}

			\section{Proof of Lemma \ref{LrMittagLeffler2}}\label{AppendixB}
			
			\begin{proof}
			We recall first that the set of constraints {\bf (i)}, {\bf (ii)} and {\bf (iii)}, imposed on the statement of Lemma \ref{LrMittagLeffler2}, 
			assures that $$\displaystyle \Re(r)>\Re\left(\frac{n-rs}{\sigma}\right)>0$$
			so that the Mellin integral formula (\ref{MellinId2}) can be indeed applied.
			
			Thereby, by setting $\eta=\sigma-s$ on both sides of eq.~(\ref{IntegralIneq}), the inequality 
			\begin{eqnarray*}
				\label{IntegralIneq2}
				\int_0^\infty\left|~\Gamma(\alpha) \rho^{\sigma-s} E_{\alpha,\alpha}(-\lambda \rho^\sigma t^{\alpha})~\right|^r\rho^{n-1}d\rho ~\leq \nonumber \\ \nonumber \\  \leq  \displaystyle \int_0^\infty \frac{\rho^{n+r(\sigma-s)-1}}{\left(1+\sqrt{\frac{\Gamma(1+\alpha)}{\Gamma(1+2\alpha)}}~\lambda \rho^\sigma t^{\alpha}\right)^{2r}} d\rho~~  (  \lambda\geq 0).
			\end{eqnarray*}
			is thus immediate.
			
			Furthermore, by direct application of eq.~(\ref{MellinId2}) we end up with
			\begin{eqnarray*}	\left(~\Gamma(\alpha)~\right)^r\int_0^\infty\left|~\rho^{\sigma-s} E_{\alpha,\alpha}(-\lambda \rho^\sigma t^{\alpha})~\right|^r\rho^{n-1}d\rho \leq D_{r,\sigma,n}^{(s)}(\alpha,\lambda)~ t^{-\frac{\alpha}{\sigma}\left({n+r(\sigma-s)}\right)}
			\end{eqnarray*}
			\begin{eqnarray*}
				= \displaystyle \frac{\Gamma\left(\frac{n+r(\sigma-s)
					}{\sigma}\right)\Gamma\left(\frac{r(\sigma+s)-n}{\sigma}\right)}{\sigma~\Gamma(2r)}\left(\sqrt{\frac{\Gamma(1+\alpha)}{\Gamma(1+2\alpha)}}~\lambda  t^{\alpha}\right)^{-\frac{n+r(\sigma-s)}{\sigma}} \displaystyle~~(\lambda\geq 0)
			\end{eqnarray*}
			and hence with (\ref{IneqD}), after few algebraic manipulations.
		\end{proof}
			
			\section{Proof of Corollary \ref{DispersiveCorollary}}\label{AppendixC}
			
			Firstly, we recall that the {\bf Assumption \ref{AssumptionF}} for $\lambda\in \mathbb{C}$ results from the closed formulae for $C_{r,\sigma,n}^{(s)}(\alpha,\beta,\lambda)$ and $D_{r,\sigma,n}^{(s)}(\alpha,\lambda)$ ($s=0$ or $s=\gamma$), obtained in Lemma \ref{LrMittagLeffler} and Lemma \ref{LrMittagLeffler2}.
			From now on, let us assume that $\displaystyle \frac{1}{p}+\frac{1}{r'}=\frac{1}{q}+1$ and $\displaystyle \frac{1}{r}+\frac{1}{r'}=1$ hold for every $1\leq p,q,r,r'\leq \infty$.

			\paragraph{\bf Proof of (1)}
			Note that $\displaystyle \frac{1}{r}=\frac{1}{p}-\frac{1}{q}$. Then, under the conditions of Proposition \ref{YoungIneqProposition}, the set of estimates (\ref{CaseIEstimates}) are thus immediate from the substitution $\displaystyle \frac{n}{r}=\frac{n}{p}-\frac{n}{q}$ on the right hand of the inequalities (\ref{YoungIneqHankel}) and (\ref{YoungIneqHankel2}), respectively.
			
			Moreover, the condition $(p,q)=(1,\infty)$ is fulfilled for $r=1$, whereas the condition $(p,q)\in \mathcal{R}_{1,2}(\sigma)$ -- see eq.~(\ref{Ogamma12}) -- results from the set of equivalences
			\begin{eqnarray*}
				\max\left\{1,\frac{n}{\sigma}\right\}< r\leq 2 &\Longleftrightarrow & \displaystyle \frac{n}{2} \leq  \frac{n}{r}< {\min\left\{n,\sigma\right\}} \\ &\Longleftrightarrow&\displaystyle \frac{n}{2} \leq \frac{n}{p} -\frac{n}{q} < {\min\left\{n,\sigma\right\}}   \\
				&\Longleftrightarrow&\frac{n}{p}-{\min\left\{n,\sigma\right\}}\displaystyle  < \frac{n}{q} \leq \frac{n}{p}-\frac{n}{2}
			\end{eqnarray*}
			and from the constraints $\displaystyle 0\leq \frac{n}{p}\leq n$ and $\displaystyle 0\leq \frac{n}{q}\leq n$.

			\paragraph{\bf Proof of  (2)}
			The proof of the set of estimates (\ref{CaseIVIIIEstimates}) results from the substitution $\displaystyle \frac{n}{r}-\gamma=\frac{n}{p}-\frac{n}{q}-\gamma$ on the right hand of the inequalities (\ref{NoSobolevYoungIneqHankel}) and (\ref{NoSobolevYoungIneqHankel2}), respectively, whereas the conditions {\bf (i)}, {\bf (ii)} and {\bf (iii)} follow straightforwardly by the set of equivalences:
			\begin{itemize}
				\item[\bf (i)] {\bf $\dot{W}^{\gamma,1}-L^\infty$ \& $\dot{W}^{\gamma,1}-\dot{W}^{\sigma,\infty}$ estimates (first case):} 
				\begin{eqnarray*}
					r=1 &\Longleftrightarrow&(p,q)=(1,\infty);
				\end{eqnarray*}
				\begin{eqnarray*}
					\sigma+\gamma>n~~\wedge~~0<\gamma<n&\Longleftrightarrow&  n-\sigma<\gamma<n.
				\end{eqnarray*}
				\item[\bf (ii)] {\bf $\dot{W}^{\gamma,p}-L^p$ \& $\dot{W}^{\gamma,p}-\dot{W}^{\sigma,q}$ estimates (second case):} 
				\begin{eqnarray*}			
					\max\left\{1,\frac{n}{\sigma+\gamma}\right\}< r\leq 2 \Longleftrightarrow & \displaystyle \frac{n}{p}-{\min\left\{n,\sigma+\gamma\right\}}\displaystyle  < \frac{n}{q} \leq \frac{n}{p}-\frac{n}{2}.
				\end{eqnarray*}
				\begin{eqnarray*}
					\sigma+\gamma>\frac{n}{2}~~\wedge~~0<\gamma<\frac{n}{2}&\Longleftrightarrow&  \frac{n}{2}-\sigma<\gamma<\frac{n}{2}.
				\end{eqnarray*}
				
				\item[\bf (iii)] {\bf $\dot{W}^{\gamma,p}-L^p$ \& $\dot{W}^{\gamma,p}-\dot{W}^{\sigma,q}$ estimates (third case):} 
				\begin{eqnarray*}
					\max\left\{1,\frac{n}{\sigma+\gamma}\right\}< r\leq\frac{n}{\gamma} \Longleftrightarrow &\displaystyle \frac{n}{p}-{\min\left\{n,{\sigma+\gamma}\right\}}\displaystyle  < \frac{n}{q} \leq \frac{n}{p}-\gamma.
				\end{eqnarray*}
				\begin{eqnarray*}
					\sigma+\gamma>\frac{n}{2}~~\wedge~~\frac{n}{2}\leq\gamma<n&\Longleftrightarrow&  \frac{n}{2}\leq\gamma<n.
				\end{eqnarray*}
			\end{itemize}

\end{document}